\documentclass[11pt]{amsart}
\usepackage{amsmath, amssymb, amsthm, enumerate}
 \usepackage[displaymath, math lines, running]{lineno}

\usepackage[left=1.5in,top=1.25in,right=1.5in,bottom=1.25in]{geometry}

\numberwithin{equation}{section}
\newtheorem{theorem}{Theorem}[section]
\newtheorem{lemma}[theorem]{Lemma}

\newtheorem{proposition}[theorem]{Proposition}
\newtheorem{corollary}[theorem]{Corollary}

\theoremstyle{definition}
\newtheorem{definition}{Definition}[section]
\newtheorem{remark}{Remark}[section]

\newcommand{\qtq}[1]{\qquad\text{#1}\qquad}

\newcommand{\cyc}[1]{\langle #1\rangle}
\newcommand{\R}{\mathbb{R}}
\newcommand{\C}{\mathbb{C}}

\newcommand{\ep}{\epsilon}
\newcommand{\grad}{\nabla}
\renewcommand{\epsilon}{\varepsilon}

\title[Scattering for 2D NLS with (full) exponential nonlinearity.]{Scattering for the two dimensional NLS with (full) Exponential Nonlinearity.}
\author{A. Adam Azzam}
\date{\today}

\begin{document}
\bibliographystyle{plain}
\maketitle
\noindent
\begin{abstract}
We obtain global well-posedness, scattering, and global $L_t^4H_{x}^{1,4}$ spacetime bounds for energy-space solutions to the energy-subcritical nonlinear Schr\"odinger equation \[iu_t+\Delta u=u(e^{4\pi |u|^2}-1)\] in two spatial dimensions. Our approach is perturbative; we view our problem as a perturbation of the mass-critical NLS to employ the techniques of Tao--Visan--Zhang from \cite{MR2354495}. This permits us to combine the known spacetime estimates for mass-critical NLS proved by Dodson \cite{Dodson:2010aa} and the work of \cite{MR2929605} and \cite{MR2569615} to prove corresponding spacetime estimates which imply scattering. 

\end{abstract}
\tableofcontents
\newpage

\section{Introduction}
We consider the Cauchy problem for a pair of defocusing nonlinear Schr\"odinger (NLS) equations on $\R^2$:
 \begin{align}&\left\{\label{PDE}\begin{array}{ll}iu_t+\Delta u=F_1(u):=u(e^{4\pi|u|^2}-1)\\ u(0,x)=u_0(x)\in H_x^1(\R^2)\end{array}\right.\\\label{PDE2} &\left\{\begin{array}{ll}iu_t+\Delta u=F_2(u):=u(e^{4\pi|u|^2}-4\pi|u|^2-1)\\ u(0,x)=u_0(x)\in H^1_x(\R^2).\end{array}\right.\end{align} Here $u:\R\times \R^2\to \C$ is a complex-valued function of time and space. In this paper, our chief interest will be to understand the long-time behavior of solutions to \eqref{PDE} and \eqref{PDE2}. Of course, before we do this, we must first clarify what we mean by a solution. 
 \begin{definition} Let $I\subseteq \R$ be an interval containing the origin. We say $u:I\times \R^2\to \C$ is a (strong) {\em solution} to \eqref{PDE} (resp. \eqref{PDE2}) if it belongs to $C_t(K; H_{x}^{1})$ for every compact interval $K\subseteq I$ and satisfies the Duhamel formula 
 
 \begin{align}
 u(t)=e^{it\Delta}u(0)-i\int_{0}^{t}e^{i(t-s)\Delta}(iu_t+\Delta u)(s)\ ds,
 \end{align} for all $t\in I$. We refer to the interval $I$ as the {\em lifespan} of $u$. We say $u$ is a global solution if $I=\R$. 
 \end{definition}
 
Solutions to \eqref{PDE} and \eqref{PDE2} conserve, respectively, the following energies:
 \begin{align}
    \label{H1}H_1(u(t))&:=\int_{\R^2}|\grad u(t,x)|^2+\tfrac{1}{4\pi}\left( e^{4\pi u^2}-1-4\pi|u|^2 \right)(t,x)\ dx, \\ \label{H2} H_2(u(t))&:=\int_{\R^2}|\grad u(t,x)|^2+\tfrac{1}{4\pi}\left( e^{4\pi u^2}-1-4\pi|u|^2-8\pi^2|u|^4 \right)(t,x)\ dx.
 \end{align}Moreover, solutions to $\eqref{PDE}$ and $\eqref{PDE2}$ both enjoy the conservation of mass \begin{align}\label{M}M(u(t)):=\int_{\R^2}|u(t,x)|^2\ dx.\end{align} When there is no chance of confusion, we will write $H_1$ for $H_1(u(t))$, $H_2$ for $H_2(u(t))$ and $M$ for $M(u(t))$. 
 
 The study of \eqref{PDE} (resp. \eqref{PDE2}) began in \cite{MR2568809}, where it was shown that global solutions exist provided $H_1(u_0)\le 1$ (resp. $H_2(u_0)\le 1$). The different techniques, estimates, and difficulties involved in the study of \eqref{PDE} and \eqref{PDE2} in the cases $H_i(u_0)<1$, $H_i(u_0)=1$, and $H_i(u_0)>1$ prompted the authors to adopt the following trichotomy.

 \begin{definition} \label{2DCriticality}We say that \eqref{PDE} (resp. \eqref{PDE2}) is {\em energy-subcritical} if $H_1(u_0)<1$ (resp. $H_2(u_0)<1$), {\em energy-critical} if $H_1(u_0)=1$ (resp. $H_2(u_0)=1$), and {\em energy-supercritical} if $H_1(u_0)>1$ (resp. $H_2(u_0)>1$).
 \end{definition}
 
Traditionally, the honorific ``energy-critical" has been given to a family of semilinear NLS in $d\ge 3$, where an available scaling symmetry leaves invariant both the energy and class of solutions. In our case neither equation enjoys a scaling symmetry, and so some explanation is needed to justify in what sense we regard \eqref{PDE} and \eqref{PDE2} as energy-critical. To properly explain this, we will begin with a brief review of the familiar energy-critical NLS in $d\ge 3$. Our goal in doing so is to understand the defining features of energy-criticality independent of scaling, and how these features manifest themselves in the theory of well-posedness. Once this is accomplished, we will draw analogies between the theory of \eqref{PDE} and \eqref{PDE2} and the theory of the energy-critical NLS in dimension $d\ge 3$. 

\subsection{Energy-Critical NLS in $\R^d$, $d\ge 3$.} In dimension $d\ge 3$, consider the defocusing nonlinear Schr\"odinger equation \begin{align}\label{nlsd3} \left\{\begin{array}{ll}iu_t+\Delta u=|u|^pu,\qquad p>0\\ u(0,x)=u_0(x)\in \dot{H}_x^s(\R^d).\end{array}\right.\end{align} Solutions to $\eqref{nlsd3}$ also enjoy the conservation of mass $\eqref{M}$ and that of  energy \begin{align*}H(u(t)):=\int_{\R^d}\frac{1}{2}|\grad u(t,x)|^2+\frac{1}{p+2}|u(t,x)|^{p+2}\ dx.
 \end{align*}
 The class of solutions to $\eqref{nlsd3}$ is invariant under the scaling \begin{align} \label{nlsscaling} u(t,x)\mapsto u_{\lambda}(t,x):= \lambda^{\frac{2}{p}}u(\lambda^2 t, \lambda x).\end{align} The effect of the scaling on the initial data is given by \[||u_{\lambda}(0)||_{\dot{H}_x^s(\R^d)}=\lambda^{s-(\frac{d}{2}-\frac{2}{p})}||u(0)||_{\dot{H}_x^s(\R^d)}.\] Thus, when $s=\frac{d}{2}-\frac{2}{p}$, the scaling $\eqref{nlsscaling}$ leaves invariant both the class of solutions and the size of the initial data. 
 
 \begin{definition} Consider the initial value problem $\eqref{nlsd3}$. Let \[s_c=\tfrac{d}{2}-\tfrac{2}{p}.\]We say the problem is {\em critical} when $s=s_c$, {\em subcritical} when $s>s_c$, and {\em supercritical} when $s<s_c$.
 \end{definition}

When $s=s_c=1$, the energy $H(u(t))$ is left invariant by the scaling $\eqref{nlsscaling}$, giving rise to the {\em energy-critical} NLS 
  \begin{align}\label{ecnls} \left\{\begin{array}{ll}iu_t+\Delta u=|u|^\frac{4}{d-2}u,\qquad d\ge 3\\ u(0,x)=u_0(x)\in \dot{H}_x^1(\R^d).\end{array}\right. \end{align} When $s=1$ and $p<\frac{4}{d-2}$ (i.e. $s_{c}<1$), we say that $\eqref{nlsd3}$ is {\em energy-subcritical}.
  
In the last two decades the energy-critical NLS \eqref{ecnls} has been the subject of a relentless, inspiring, and successful campaign to understand the local and global behavior of its solutions. Local well-posedness was first proved by Cazenave and Weissler in \cite{MR1055532}, who showed that the length of a solution's lifespan depends on the profile of $u_0$ rather than the norm of $u_0$. The first victory on the front of large-data global well-posedness was made by J. Bourgain in \cite{MR1626257}, who proved global well-posedness and exhibited global spacetime bounds for spherically symmetric initial data in $d=3,4$ by introducing what is now known as the `induction on energy' paradigm. Using this paradigm, and introducing a wealth of new ideas and techniques, the authors of \cite{MR2415387} managed to remove Bourgain's assumption of spherically symmetric data in $d=3$. Adapting these techniques to handle newfound difficulties in high dimensions, the problem was completely resolved by the authors of \cite{MR2288737} $(d=4)$ and \cite{MR2318286} $(d\ge 5)$. 

Though the notion of energy-criticality in $d\ge 3$ is defined through an available scaling symmetry, it is important to understand the characteristic features of the energy-subcritical and energy-critical nonlinearities independent of this symmetry. 

As \eqref{nlsscaling} reveals, energy-criticality is determined by the response of the $\dot{H}_x^1$ norm at fine scales or, equivalently, high-frequencies. In the energy-critical case, the kinetic and potential energy norms are equally strong at all scales. In the energy-subcritical case, the kinetic energy norm dominates the potential energy norm at fine scales. This phenomenon is expressed concisely in the Sobolev inequality, which in $d\ge 3$ reads
\begin{align}\label{endpointsobolev}||f||_{L_x^{\frac{2d}{d-2}}}&\le C_{d}||\grad f||_{L_x^2},\\ \label{sobembhid}\sup_{||u||_{H_{x}^{1}}\le 1}||u||_{L_x^p}&\le C(p,d)\qtq{}2\le p\le \frac{2d}{d-2}.\end{align}  Heuristically, Sobolev embedding informs us of how strong our nonlinearity may be, i.e. how large $p$ may be, before the potential energy norm overpowers the kinetic energy norm. 

These features present themselves in the well-posedness theory in a few ways. In the energy-subcritical setting the time of existence guaranteed in the local theory depends only on the size of the initial data; in the energy-critical setting, the time of existence depends on its profile. The local theory may be iterated to extend the lifespan of a local solution provided there is no energy concentration.  A computation shows that 
\[H(u_{\lambda})=\lambda^{2(1-s_c)}H(u)\to \infty,\]which converges to $\infty$ as $\lambda\to \infty$ in the energy-subcritical case $s_c<1$. Thus, energy conservation discourages concentration in this case. In the energy-critical setting $s_c=1$, however, the conservation of energy does not immediately rule out the possibility of concentration. 
 
 \subsection{Energy-Critical NLS in $\R^2$}
 
 In $\R^2$, Sobolev embedding guarantees that every polynomial nonlinearity is energy-subcritical. Indeed, if $u\in H_{x}^{1}$ is localized to frequency $\sim N$ and $||u||_{H_{x}^1(\R^2)}\le 1$ (say), then Bernstein's inequalities show that \[||u||_{L_x^{p}}^{p}\lesssim_{p} N^{-2}||\grad u||_{L_x^2}^{2}.\] At fine scales $(N\gg 1)$, we see that the kinetic energy dominates the potential energy. Thus, if we are to identify an energy-critical nonlinearity, it is natural to consider an exponential nonlinearity.
 
In \cite{MR2568809} and \cite{MR2929605}, the authors identified the nonlinearities $F_1(u)$ and $F_2(u)$ in $\eqref{PDE}$ and $\eqref{PDE2}$ as energy-critical when $H_1(u_0)=1$ and $H_2(u_0)=1$, respectively, using a substitute for the end-point Sobolev inequality \eqref{endpointsobolev} known as the Moser--Trudinger inequality. We will discuss the Moser--Trudinger inequality and its many variants in detail below.  In analogy to the end-point Sobolev embedding \eqref{sobembhid} in $d\ge 3$, the Moser--Trudinger inequality informs us of the exact speed at which a nonlinearity may grow before the potential energy term overpowers the kinetic energy norm. Unlike the end-point Sobolev embedding however, the Moser--Trudinger inequality holds only for functions with sufficiently small $\dot{H}_x^1$ norm. Thus, in the case $d=2$, we fix one nonlinearity $F_1(u)$ or $F_2(u)$ and classify energy-criticality depending on the size of initial energy.

In \cite{MR2568809}, the authors established a complete trichotomy analogous to the energy-critical case in $d\ge 3$, corresponding to the size of the initial data's energy. In the terminology of Definition \ref{2DCriticality}, the Cauchy problems \eqref{PDE} and \eqref{PDE2} are globally well-posed in the energy-subcritical and energy-critical cases. In analogy to the $d\ge 3$ case, the lifespan of a local solution in the energy-subcritical cases of \eqref{PDE} and \eqref{PDE2} depend on the size of the initial data, whereas the lifespan in the energy-critical cases depend fully on the profile of the initial data. Moreover, they demonstrate that \eqref{PDE} and \eqref{PDE2} are ill-posed for a subset of initial data in the energy-supercritical case (though no ill-posedness results are known for slightly energy-supercritical data). 

With global well-posedness established in the energy-critical and energy-subcritical cases, the next natural question to investigate was whether global solutions scatter in ${H}_x^1$. In \cite{MR2929605}, the authors proved the existence of global spacetime bounds which imply scattering for global solutions to \eqref{PDE2} in the energy-subcritical case. The key insight was to use the {\em a-priori} Morawetz estimate established independently by Colliander--Grillakis--Tzirakis in \cite{MR2527809} and Planchon--Vega in \cite{MR2518079}. In \cite{MR3161603}, the authors expanded on \cite{MR2929605}, by proving the existence of global spacetime bounds which imply scattering for global solutions in the energy-critical case of \eqref{PDE2} under the additional assumption of radial initial data. 

The chief difficulty in establishing similar results for the corresponding cases of \eqref{PDE} stem from the poor decay of the cubic term in the Taylor expansion of $F_1(u)$:
\[F_1(u)=4\pi u|u|^2+8\pi^2u|u|^4+\cdots.\] Indeed, $F_1(u)$ can only decay at least as slow as the cubic nonlinearity $4\pi u|u|^2$ does.  Thus, we may only expect global solutions to \eqref{PDE} to scatter in $H_x^1$ if we expect scattering in $H_x^1$ for global solutions to the the associated Cauchy problem
 \begin{align}\label{h1mcnls} \left\{\begin{array}{ll}iu_t+\Delta u=4\pi|u|^2u\\ u(0,x)=u_0(x)\in H_x^1(\R^2).\end{array}\right.\end{align} In $\R^2$, this corresponds to the mass-critical NLS, whose theory we review briefly below. 
 \subsection{Mass-Critical NLS in $\R^2$}
 When $s=s_c=0$ the mass $M(u(t))$ is left invariant by the scaling ($\ref{nlsscaling})$, giving rise to the {\em mass-critical} NLS. In two space dimensions, this takes the form:
 \begin{align}\label{mcnls} \left\{\begin{array}{ll}iu_t+\Delta u=g(u):=|u|^2u\\ u(0,x)=u_0(x)\in L_x^2(\R^2).\end{array}\right.\end{align}  
 
 The local theory for \eqref{mcnls} was established by Cazenave and Weissler in \cite{MR1030060}. Analogously to the local theory for the energy-critical case in $d\ge 3$, the lifespan of the local solutions they constructed depended on the profile of the initial data and not just the $L_x^2$-norm. Two decades later, in \cite{MR2557134}, Killip--Tao--Visan showed that for radial initial data, \eqref{mcnls} is globally well-posed and that solutions obey global spacetime bounds; in particular, scattering holds. Soon after, in \cite{Dodson:2010aa}, Dodson removed the radiality assumption and established the theorem in its full generality. 
\begin{theorem}[Radial \cite{MR2557134}, Non-Radial \cite{Dodson:2010aa}]\label{dodson} For $u_0\in L_x^2$, there exists a unique strong solution $u:\R\times \R^2\to \C$ to $\eqref{mcnls}$. Moreover, $u$ satisfies \[\int_{\R}\int_{\R^2}|u(t,x)|^4\ dx\ dt<C(||u_0||_{L_x^2})\] and scatters both backwards and forwards in time. 
\end{theorem}
Thus, when $u_0\in L_x^2$ the global solution $u:\R\times \R^2\to \C$ scatters in $L_x^2$. Returning to \eqref{h1mcnls}, we would like to know, for $H_x^1$ data, whether the global solution guaranteed by Theorem \ref{dodson} scatters in $H_x^1$. Luckily, this follows from a standard lemma (See, for example, Lemma 3.10 in \cite{MR2354495}). 
\begin{lemma}[Persistence of Regularity] \label{STBD}Let $k=0,1$ and $I$ be a compact time interval. Let $v$ be the unique solution to $\eqref{mcnls}$ on $I\times \R^2$ with \begin{align}
||v||_{L_t^4L_x^4(I\times \R^2)}\le L\label{STBD1}.
\end{align} Then, if $t_0\in I$ and $v(t_0)\in H_x^k$, we have \begin{align}
||v||_{S^k(I\times \R^2)}\le C(L)||v(t_0)||_{{H}_x^k}.\end{align}
\end{lemma}

\subsection{Main Results}
In this paper we address the question of whether global solutions to \eqref{PDE} obey global spacetime bounds and scatter in ${H}_x^1$ in the energy-subcritical case. To do so, we exploit the insights of Tao-Visan-Zhang from \cite{MR2354495}. 

In \cite{MR2354495}, the authors embark on a systematic study of Cauchy problems of the form 
 \begin{align}\label{combined} \left\{\begin{array}{ll}iu_t+\Delta u=|u|^{p_1}u+|u|^{p_2}u\\ u(0,x)=u_0(x)\in H_x^1,\end{array}\right.\end{align} with $u:\R_{t}\times \R_{x}^{d}\to \C$ and $0<p_1<p_2\le \frac{4}{d-2}$. They show that for various values of $p_1$ and $p_2$, if $\eqref{combined}$ is globally well-posed then a solution to \eqref{combined} can be viewed as a perturbation of a solution to
  \begin{align}\label{combined2}\left\{\begin{array}{ll}iu_t+\Delta u=|u|^{p_1}u\\ u(0,x)=u_0(x)\in H_x^1,\end{array}\right.\end{align} with error $|u|^{p_2}u$.  In these cases, they show, the solution to \eqref{combined} may inherit global spacetime bounds, if they exist, for \eqref{combined2}. 
  
Our approach is similar. We write \[F_1(u)=4\pi |u|^2u+F_2(u),\] and thus view \eqref{PDE} as a perturbation of the mass-critical NLS \eqref{h1mcnls} with error $F_2(u)$. We exploit the estimates from \cite{MR2929605} in the energy-subcritical case of \eqref{PDE2} to show that the error term $F_2(u)$ is mild enough to derive global spacetime bounds from those enjoyed by global solutions of the mass-critical problem. These results are summarized in the following theorem. 

\begin{theorem} \label{perturbtheorem}For $u_0\in H_x^1$ with $H_1(u_0)\le 1$, there exists a unique strong solution $u:\R\times \R^2\to \C$ to $\eqref{PDE}$. If $H_1(u_0)<1$, then $u$ satisfies \[\int_{\R}\int_{\R^2}|u(t,x)|^4+|\grad u(t,x)|^4\ dx\ dt<C(||u_0||_{H_x^1})\] and scatters both backwards and forwards in time. 
\end{theorem}

To establish spacetime bounds on global solutions to \eqref{PDE} or \eqref{PDE2}, the Strichartz inequality (see Section \ref{prelim}) informs us that we need only investigate which spacetime bounds are available to estimate the nonlinearity. This requires an inquiry into which $1\le p,q\le \infty$ do we have control over \begin{align}\label{nonlinestexp}||F_i(u)||_{L_t^pL_x^q(\R\times \R^2)},\end{align} for $i\in \{1,2\}$. Indeed, the analysis in \cite{MR3161603},\cite{MR2568809}, and \cite{MR2929605} relied crucially on estimating terms like \eqref{nonlinestexp}. Our second main result describes a wide range of exponents for which one may control \eqref{nonlinestexp}. 

\begin{theorem}[Moser--Trudinger--Strichartz] \label{GMTS}Let $u:\R\times \R^2\to \C$ satisfy \begin{align}\label{cond1}||u(t)||_{\dot{H}_x^1(\R^2)}\le 1\qtq{and}||u(t)||_{L_x^2(\R^2)}^2=M,\end{align} for every $t\in \R$. If $s\ge 4$ and $1\le p,q\le \infty$ are such that \begin{align}\label{cond2}\frac{q'}{p}> 1,\end{align}then \begin{align}\label{GMTSC}\bigl\| |u|^se^{4\pi |u|^2}\bigl\|_{L_t^pL_x^q(\R\times \R^2)}\lesssim_{M} ||u||_{L_t^4H_{x}^{1,4}(\R\times \R^2)}^{\frac{4}{p}}.\end{align} Conversely, if $\eqref{GMTSC}$ holds for all $u$ satisfying $\eqref{cond1}$, then $q$ and $p$ must satisfy $\eqref{cond2}.$ 
\end{theorem}

It is important to understand that Theorem \ref{GMTS} holds for arbitrary spacetime functions satisfying \eqref{cond1}. Though every solution to \eqref{PDE} and \eqref{PDE2} satisfies \eqref{cond1} in the energy-subcritical and energy-critical cases, it is obvious that not every spacetime function satisfying \eqref{cond1} solves \eqref{PDE} or \eqref{PDE2}. Thus, improvements into the range of exponents may be made if we restrict ourselves to only consider solutions to \eqref{PDE} and \eqref{PDE2}. For example, in \cite{MR3161603}, the authors demonstrate that one may improve \eqref{cond2} if $u(t)$ was sufficiently small in an Orlicz space for all times $t\in \R$.

The endpoint in \eqref{PDE2} represents a frustrating obstacle in extending Theorem \ref{perturbtheorem} to the energy-critical case of \eqref{PDE}. Indeed, even assuming conditionally the existence of global spacetime bounds in \eqref{PDE2}, we are unable to extend the results to similar ones for \eqref{PDE}. This obstacle arises in the perturbation theory, when one has to estimate a dual Strichartz norm of the form \[||\grad F_2(u)||_{L_t^\frac{2}{1+2\delta}L_x^{\frac{1}{1-\delta}}},\]for $0<\delta\le \frac{1}{2}$. At times when $||u(t)||_{L_x^\infty(\R^2)}$ is large, this term behaves like \[||\grad u |u|^4e^{4\pi |u|^2}||_{L_t^\frac{2}{1+2\delta}L_x^{\frac{1}{1-\delta}}}.\] To close our perturbation argument, we may only estimate $\grad u$ in $L_t^\infty L_x^2$ or $L_t^4L_x^4$. By interpolation and H\"older this requires one to estimate the term $|u|^4e^{4\pi |u|^2}$ in a space $L_t^pL_x^q$ with $\frac{q'}{p}=1$, the unavailable end-point case of Theorem \ref{GMTS}.

The investigation into which exponents in \eqref{cond2} are permissible for solutions to \eqref{PDE} or \eqref{PDE2} is an important line of future inquiry. 

\section{Acknowledgements} It is difficult for me to articulate just how much I appreciate the guidance and support of my advisors, Rowan Killip and Monica Visan. I am incredibly grateful to them for introducing me to this problem, for their time in discussing it with me, and for carefully reading this manuscript. This work was supported in part by NSF grant DMS 1265868 (P.I. Rowan Killip) and NSF grant DMS-1500707 (P.I. Monica Visan).

\section{Preliminaries}
\label{prelim}
We begin by fixing some notation. We will write $X\lesssim Y$ if there exists a constant $C$ so that $X\le CY$. When we wish to stress the dependence of this implicit constant on a parameter $\ep$ (say), so that $C=C(\ep)$, we write $X\lesssim_{\ep}Y$. We write $X\sim Y$ if $X\lesssim Y$ and $Y\lesssim X$. If there exists a {\em small} constant $c$ for which $X\le cY$ we will write $X\ll Y$. 

For $1\le r<\infty$ we recall the Lebesgue space $L_x^r(\R^2)$, which is the completion of smooth compactly supported functions $f:\R^2\to \C$ under the norm \[||f||_{r}=||f||_{L_x^r(\R^2)}:=\left({\int_{\R^2}|f(x)|^r\ dx}\right)^{\frac{1}{r}}.\] When $r=\infty$, we employ the essential supremum norm. For $f:\R\times \R^2\to \C$ we use $L_t^pL_x^q$ to denote the spacetime norm \[||f||_{p,q}=||f||_{L_t^pL_x^q(\R\times \R^2)}=\left(\int_{\R}\left(\int_{\R^2}|f(t,x)|^q\ dx\right)^{\frac{p}{q}}\ dt\right)^{\frac{1}{p}}\] with the natural modifications when either $p$ or $q$ is infinity, or when $\R\times \R^2$ is replaced by some other spacetime region. In particular, on the spacetime slab $[-T,T]\times \R^2$ we write \[||f||_{L_t^pL_x^q([-T,T]\times \R^2)}=||f||_{L_T^pL_x^q}.\]
When $q=p$ we write $L_{t,x}^p=L_{t}^pL_x^p$.

Our convention for the Fourier transform on $\R^2$ is \[\hat{f}(\xi)=\frac{1}{2\pi}\int_{\R^2}e^{-ix\cdot \xi}f(x)\ dx.\] The Fourier transform allows us to define the fractional differentiation operators \[\widehat{|\grad|^sf}(\xi)=|\xi|^s\hat{f}(\xi)\qquad \text{ and }\qquad \widehat{\langle\grad \rangle^sf}(\xi)=(1+|\xi|^2)^\frac{s}{2}\hat{f}(\xi).\] The fractional differentiation operators give rise to the (in-)homogenous Sobolev spaces. We define $\dot{H}^{1,r}(\R^2)$ and $H^{1,r}(\R^2)$ to be the completion of smooth compactly supported functions $f:\R^2\to \C$ under the norms \[||f||_{\dot{H}^{1,r}(\R^2)}=|||\grad | f||_{L_x^r}\qquad \text{ and }\qquad ||f||_{{H}^{1,r}(\R^2)}=|||\cyc{\grad} f||_{L_x^r}.\] When $r=2$, we simply write $H^{1,r}=H^{1}$ and $\dot{H}^{1,r}=\dot{H}^1$. 

Throughout this paper, we will often need to dampen the mass term in the Sobolev norm and so for $0<\mu\le 1$ we define \[||u||_{H_{\mu}^1}^2=\mu ||u||_{L_x^2}^2+||\grad u||_{L_x^2}^2.\]

In two dimensions, we say that a pair of exponents $(q,r)$ is Schr\"odinger-{\em admissible} if $\frac{1}{q}+\frac{1}{r}=\frac{1}{2}$ and $2\le q,r\le \infty$ but $(q,r)\ne (2,\infty)$. We say that a pair of exponents $(q,r)$ is {\em dual} Schr\"odinger-admissible if their H\"older conjugates $(q',r')$ are Schr\"odinger-admissible. It is straightforward to show that $(q,r)$ is dual Schr\"odinger-admissible if and only if $\frac{1}{q}+\frac{1}{r}=\frac{3}{2}$ and $(q,r)\ne (2, 1)$. If $I\times \R^2$ is a spacetime slab, we define the $S^0(I\times \R^2)$ {\em Strichartz norm} by \[||u||_{S^0(I\times \R^2)}:=\sup_{(q,r)\text{ admissible }} ||u||_{L_t^q L_x^r(I\times \R^2)}.\] The attentive reader will notice that since we are in two dimensions, we need to restrict the supremum to a closed subset of admissible pairs (as to avoid the inadmissible endpoint). Every argument in this paper requires finitely many admissible pairs, so this caveat matters little to us. Similarly, we define the $S^1(I\times \R^2)$ and $S(I\times \R^2)$ norms to be \[||u||_{S^1(I\times \R^2)}:= ||\grad u||_{S^0(I\times \R^2)}\qtq{and}||u||_{S(I\times \R^2)}:=||\langle \grad \rangle u||_{S^0(I\times \R^2)}.\]  We also use $N^0(I\times \R^2)$ to denote the dual space of $S^0(I\times \R^2)$ and \[N^1(I\times \R^2):=\{u;\, \grad u\in N^0(I\times \R^2)\}.\] As before, we define the $N(I\times \R^2)$ norm to be \[||u||_{N(I\times \R^2)}=||\cyc{\grad}u||_{N^0(I\times \R^2)}.\]

\begin{lemma}[Strichartz Estimates, \cite{MR1646048}] Let $I$ be a compact time interval, $k\in \{0,1\},$ and let $u:I\times \R^2\to \C$ be a solution to the Schr\"odinger equation \[iu_t+\Delta u=F,\] for a function $F$. Then \[|| |\grad|^k u||_{S^0(I\times \R^2)}\lesssim ||u(t_0)||_{\dot{H}^k(\R^2)}+|||\grad|^k F||_{N^0(I\times \R^2)}.\]
\end{lemma}

\subsection{Pointwise Estimates} In this subsection, we record pointwise estimates needed for the well-posedness and perturbation theory in Sections $4$ and $5$, respectively. We assume that $F:\C\to \C$ is continuously differentiable, and use $\partial_z$ and $\partial_{\bar{z}}$ to denote the usual Wirtinger derivatives
\begin{align*}\partial_zF=F_z:=\tfrac{1}{2}\left(\tfrac{\partial F}{\partial x}-i\tfrac{\partial F}{\partial y}\right)\qtq{and}\partial_{\bar{z}}F=F_{\bar{z}}:=\tfrac{1}{2}\left(\tfrac{\partial F}{\partial x}+i\tfrac{\partial F}{\partial y}\right).\end{align*} The Fundamental Theorem of Calculus permits us to write
\begin{align}\label{ptwdiff}F(z)-F(w)=\int_{0}^{1}(z-w)F_z(w+\theta(z-w))+\overline{(z-w)}F_{\bar{z}}(w+\theta(z-w))\ d\theta. \end{align}We may use $\eqref{ptwdiff}$ to bound \begin{align}\label{ptwdiff2}|F(z)-F(w)|\lesssim |z-w|( |F_{z}(z)|+|F_{z}(w)|+|F_{\bar{z}}(z)|+|F_{\bar{z}}(w)|).\end{align} Of course, if $F$ is continuously twice differentiable we may obtain $\eqref{ptwdiff2}$ for $F_z$ and $F_{\bar{z}}$ in place of $F$.

 Recall our notation: \[F_1(z)=z(e^{4\pi|z|^2}-1),\qquad F_2(z)=z(e^{4\pi |z|^2}-4\pi|z|^2-1),\qquad g(z)=z|z|^2.\]  
\begin{lemma} If $z_1,z_2\in \C$, then 
\begin{align} &\label{cubicdiff}|g(z_1)-g(z_2)|\lesssim |z_1-z_2| \sum_{j=1,2}|z_j|^2\\ 
&\label{cubicderivdiff}|\partial_zg(z_1)-\partial_zg(z_2)|\lesssim |z_1-z_2|\sum_{j=1,2}|z_j|\\
&\label{F1diff}|F_1(z_1)-F_1(z_2)|\lesssim |z_1-z_2|\sum_{j=1,2}e^{4\pi |z_j|^2}-1+|z_j|^2e^{4\pi |z_j|^2}\\
&\label{F2diff}|F_2(z_1)-F_2(z_2)|\lesssim |z_1-z_2|\sum_{j=1,2}e^{4\pi|z_j|^2}-4\pi |z_j|^2-1+|z_j|^2(e^{4\pi|z_j|^2}-1)\\ &\label{F1derivdiff}|\partial_z F_1(z_1)-\partial_z F_1(z_2)|\lesssim |z_1-z_2|\sum_{j=1,2}|z_j|e^{4\pi|z_j|^2}+|z_j|^3e^{4\pi|z_j|^2}\\
&\label{F2derivdiff}|\partial_zF_2(z_1)-\partial_zF_2(z_2)|\lesssim |z_1-z_2|\sum_{j=1,2} |{z_j}|(e^{4\pi |z_j|^2}-1)+|z_j|^3e^{4\pi|z_j|^2}.
\end{align} Finally, for $i\in \{1,2\}$ and any $\ep>0$ we have 
\begin{align}  
\label{F1diff2}|F_i(z_1)-F_i(z_2)|&\lesssim_{\ep} |z_1-z_2|\sum_{j=1,2}\big(e^{4\pi(1+\ep) |z_j|^2}-1\big)\\
\label{F1derivdiff2}|\partial_z F_i(z_1)-\partial_z F_i(z_2)|&\lesssim_{\ep} |z_1-z_2|\sum_{j=1,2}\big(|z_j|+e^{4\pi(1+\ep)|z_j|^2}-1\big).
\end{align}Moreover, $\eqref{F1derivdiff}$, $\eqref{F2derivdiff}$, and $\eqref{F1derivdiff2}$ hold with $\partial_{\bar{z}}$ in place of $\partial_z$. 

\end{lemma}

\subsection{Endpoint Sobolev Inequalities}
In two space dimensions, Sobolev embedding guarantees that for each $2\le p<\infty$ or $2<r\le \infty$ there exists some $c_r,c_p>0$ so that 
\begin{align}
||u||_{L_x^p(\R^2)}&\le c_{p}||u||_{H_x^1(\R^2)} \label{sob}.\\ 
||u||_{L_x^\infty(\R^2)}&\le c_{r}||u||_{H_x^{1,r}(\R^2)} \label{sob2} .
\end{align} 

The failure of $\eqref{sob}$ to hold at the endpoint $p=\infty$ and of $\eqref{sob2}$ to hold with $r=2$ has motivated mathematicians over the last half century to find appropriate substitutes. This has lead to two lines of inquiry, essentially dual to one another, and with essentially dual answers. The first line of inquiry attempts to discover for which functions $f:\R\to \R$ we have \begin{align}\label{1I}\sup_{||u||_{H_x^1}\le 1}\int_{\R^2}f(|u|)\ dx<\infty.\end{align}  The second line of inquiry attempts to discover for which functions $g:\R\to \R$ we have \begin{align}\label{2I}||u||_{L_x^\infty(\R^2)}\le g(||u||_{H_x^{1,r}(\R^2)}),\end{align} for $r>2$ and $u\in H_x^1$. 

The first line of investigation lead to the family of Moser--Trudinger inequalities which show that $\eqref{1I}$ holds for a variety of functions of square exponential growth. The second line of investigation lead to the family of Br\'ezis--Wainger--Gallou\"et inequalities, which show that $\eqref{2I}$ holds for a variety of functions of $\sqrt{\log}$ growth. 

\subsubsection{Moser--Trudinger Inequalities} In \cite{MR0216286}, Trudinger observed that one could take $c_p=c\cdot \sqrt{p}$ in $\eqref{sob}$ for some constant $c$ independent of $p$. By expanding into power series, he found that \[||e^{\alpha |u|^2}-1||_{L_x^1}=\sum_{k=1}^{\infty}\tfrac{\alpha^k ||u||_{2k}^{2k}}{k!}\le \sum_{k=1}^{\infty}\tfrac{(2kc^2\alpha)^k}{k!}<\infty\] for $0<\alpha\ll 1$ sufficiently small. A few years later, in \cite{MR0301504}, Moser applied symmetrization techniques and found that the constant $\alpha=4\pi$ was optimal. Since then, mathematicians have discovered a wide variety of Moser--Trudinger(-type) inequalities, which are indispensable in our analysis. In this paper we rely on a few, which are listed below.

In the energy-subcritical setting, we see that the exponential nonlinearity behaves like the first nonzero term in its Taylor approximation. 
 
\begin{proposition}[\cite{MR1646323}]For each $\alpha\in [0,4\pi)$ there exists $c=c(\alpha)$ so that 
\label{MT1}
\begin{align}||e^{\alpha |u|^2}-1||_{L_x^1(\R^2)}\le c||u||_{L_x^2(\R^2)}^2,\end{align}
uniformly for $u\in H_{x}^1(\R^2)$ with $||\grad u||_{L_x^2(\R^2)}\le 1$. 
\end{proposition}

\begin{proposition}[\cite{MR3161603}] For each $\alpha\in [0,4\pi)$ and $s\ge 1$, there exists a constant $c=c(\alpha,s)$ so that 
\label{MT4}
\begin{align}|||u|^s e^{\alpha |u|^2}||_{L_x^1(\R^2)}\le c||u||_{L_x^s}^s,\end{align}
uniformly for $u\in H_x^1(\R^2)$ with $||\grad u||_{L_x^2(\R^2)}\le 1$. 
\end{proposition}

In the energy-critical setting, the previous proposition fails if we only require control over the $\dot{H}^1$ norm. However, if we require the full $H_x^1$ norm to be sufficiently small, we recover a useful substitute. 

\begin{proposition}[\cite{MR2109256}] There exists a constant $c$ so that 
\label{MT2}
\begin{align}||e^{4\pi |u|^2}-1||_{L_x^1(\R^2)}\le c,\end{align}
uniformly for $u\in H_x^1(\R^2)$ with $||u||_{H_x^1(\R^2)}\le 1$. 
\end{proposition}

A natural question arises from the previous concession, namely: what is the best bound we may obtain by requiring only that the $\dot{H}^1$ norm be sufficiently small. This is answered in the following proposition.

\begin{proposition}[\cite{Ibrahim:2011aa}] There exists a constant $c$ so that 
\label{MT3}
\begin{align}\int_{\R^2}\frac{e^{4\pi |u|^2}-1}{(1+|u|)^2}\ dx\le c||u||_{L_x^2}^2,\end{align}
uniformly for $u\in H_x^1(\R^2)$ with $||\grad  u||_{L_x^2(\R^2)}\le 1$. 
\end{proposition}

\subsubsection{Br\'ezis--Wainger--Gallou\"et Inequalities}

In our analysis we will see that the most difficult part in controlling the nonlinearities in $\eqref{PDE}$ and $\eqref{PDE2}$ are those times when $u(t)$ is large in $L_x^\infty$. We have two estimates in our arsenal to control the $L_x^\infty$ norm, the latter of which is more powerful. The first is the Morrey Embedding $\eqref{sob2}$, and the second is the sharp Br\'ezis--Wainger--Gallou\"et inequality stated in Proposition $\ref{SBWG}$. Though both will prove useful, the strength of Proposition $\ref{SBWG}$ over $\eqref{sob2}$ comes from our knowledge of the explicit constant.

  \begin{proposition}[Sharp Br\'ezis--Wainger--Gallou\"et, \cite{MR2578540,MR2280178}] \label{SBWG}Let $2< r<\infty$ and $0<\mu \le 1$. Then for any $\alpha>1$ there exists a constant $C_{\alpha}$ depending on $\alpha$ so that 
     \begin{align}\label{bwg}||u||_{L_x^\infty}^2\le \frac{\alpha}{4\pi}\frac{2r}{r-2}||u||_{H_{\mu}^1}^2\log\left[C_{\alpha}+\frac{c_{r} (8/\mu)^{1-\frac{2}{r}}||u||_{H_x^{1,r}}}{||u||_{H_{\mu}^1}}\right],\end{align} where $c_{r}$ is the constant appearing in $\eqref{sob2}$.
 \end{proposition}

 \subsection{Morawetz Estimates}
The Morawetz estimate, proved independently by Colliander--Grillakis--Tzirakis and Planchon--Vega, is essential to our analysis by providing an {\em a-priori} estimate to start from and build upon.
\begin{lemma}{\cite{MR2527809, MR2518079}}\label{morawetz} If $u:\R\times \R^2\to \C$ is a global solution of $\eqref{PDE}$, $\eqref{PDE2}$, or $\eqref{mcnls}$ then \[||u||_{L_t^4L_x^8(\R\times \R^2)}\lesssim ||u||_{L_t^\infty L_x^2(\R\times \R^2)}^{\frac{3}{4}}||\grad u||_{L_x^\infty L_x^2(\R\times \R^2)}^{\frac{1}{4}}\lesssim ||u||_{L_t^\infty H_x^1(\R\times \R^2)}.\]

\end{lemma}

\section{Moser--Trudinger--Strichartz Inequalities}
The Moser--Trudinger and Br\'ezis--Wainger--Gallou\"et inequalities are invaluable when proving $L_x^q$-estimates involving the exponential nonlinearities appearing in $\eqref{PDE}$ and $\eqref{PDE2}$. Indeed, one can (and should) regard the Moser--Trudinger inequality as an $L_x^1$ estimate on the nonlinearity and the Br\'ezis--Wainger--Gallou\"et inequality as an $L_x^\infty$ estimate, permitting one to interpolate in between. This technique has been successfully employed in \cite{MR2568809} and \cite{MR2929605}. 

In this section, we systematically study which {\em spacetime} bounds are available in this setting. Specifically, we prove a general family of spacetime bounds including, to the author's knowledge, the first that hold generally in the energy-critical case. Several instances of these estimates are implicit in the literature (see \cite{MR3161603},\cite{MR2568809},\cite{MR2929605}). We contend that this formulation will streamline the standard proofs of well-posedness and scattering in the energy-subcritical cases; see Sections 4 and 5. 

\subsection{Global Moser--Trudinger--Strichartz} The lack of an {\em a priori} $L_x^\infty$ bound on solutions to $\eqref{PDE}$ and $\eqref{PDE2}$ present a frustrating but manageable challenge in proving spacetime estimates. We can overcome this challenge by splitting the set of times into two parts. If $t\in \R$ satisfies $||u(t)||_{L_x^\infty}\le K$ (say), then we use the trivial bound\[|u(t,x)(e^{4\pi|u(t,x)|^2}-4\pi |u(t,x)|^2-1)|\lesssim_{K} |u(t,x)|^5.\] As we will see, this contribution is easily estimated in two dimensions. The enemy, then, in establishing good spacetime bounds are those times when the $||u(t)||_{L_x^\infty}$ norm is large. We begin with a trivial $L_x^1$ estimate, which, when combined with a hard-fought $L_t^pL_x^\infty$ estimate, will yield the full Moser--Trudinger--Strichartz estimate. 
\begin{lemma}[An $L_x^1$ estimate] \label{L1EST} If $u:\R\times \R^2\to \C$ satisfies $||u||_{L_t^\infty\dot{H}_x^1}\le 1$, then \[||e^{4\pi |u(t)|^2}-1||_{L_x^1(\R^2)}\lesssim ||u(t)||_{L_x^2(\R^2)}^2(1+||u(t)||_{L_{x}^{\infty}}^2).\] 
\begin{proof} As $u$ satisfies the hypotheses of Proposition $\ref{MT3}$, we know that \[\int_{\R^2}\frac{e^{4\pi |u(t,x)|^2}-1}{(1+|u(t,x)|)^2}\ dx\lesssim ||u(t)||_{L_x^2}^2.\] Thus, we see that \[\int_{\R^2} e^{4\pi |u(t,x)|^2}-1\ dx\lesssim ||u(t)||_{L_x^2(\R^2)}^2(1+||u(t)||_{L_x^\infty})^2\lesssim ||u(t)||_{L_x^2(\R^2)}^2(1+||u(t)||_{L_x^\infty}^2),\] as desired.\end{proof}
\end{lemma}

Although the proof of Lemma \ref{L1EST} is elementary, it further illustrates our point that the nonlinearity is well controlled at times when $||u(t)||_{L_x^\infty}$ is small. The next lemma will demonstrate how to handle those times when $||u(t)||_{L_x^\infty}$ is large. 

\begin{lemma}[An $L_t^pL_x^\infty$ estimate] \label{LEST}Let $\alpha>0$ and suppose $u:\R\times \R^2\to \C$ is such that \[||u||_{L_t^\infty \dot{H}_x^1}\le 1\qtq{and}||u(t)||_{L_x^2(\R^2)}^2=M\] for every $t\in \R$. Then for all $1\le p<\infty$ with \[||u||_{L_t^\infty\dot{H}_x^1}<\frac{1}{p\alpha}\] there exists some $K>0$ so that \begin{align}||e^{4\pi \alpha |u|^2}||_{L_t^pL_x^\infty(I\times \R^2)}\le C(M)||u||_{L_t^4H_{x}^{1,4}}^\frac{4}{p},\end{align} where $I=\{t\in \R: ||u(t)||_{L_x^\infty(\R^2)}>K\}$. 

\begin{proof} We begin by defining a few relevant parameters,  which will help us in interpolating later in the proof and avoiding an unnecessary discussion of the dependence of some constants on others.  

We first choose $0<\mu<1$ sufficiently small so that \begin{align}\label{critineq1}\frac{1}{p\alpha}>||u||_{L_t^\infty H_{\mu}^1}^2.\end{align} 

We then choose an increasing continuous function $f:(2,\infty)\to (1,\infty)$ so that \begin{align}\label{flimit1}&\lim_{r\to 2^{+}}f(r)=1,\qquad \lim_{r\to \infty}f(r)=\infty,\qquad \text{and} \\  \label{fbound1} &\sup_{r\in (2,\infty)}\frac{r-2}{2r} \cdot \frac{4}{p\alpha}\cdot \frac{1}{f(r)}\cdot \frac{1}{||u||_{L_t^\infty H_\mu^1}^2}<1.\end{align}

Now, for $2<r<\infty$, define \begin{align}\label{thetadef}\lambda(r)&:=\tfrac{f(r)}{4\pi}\cdot \tfrac{2r}{r-2}\qquad \text{and}\qquad \\ \theta(r)&:=\tfrac{1}{p\alpha}\cdot \tfrac{4}{4\pi \lambda(r)||u||_{L_t^\infty H_\mu^1}^2}=\tfrac{r-2}{2r} \cdot \tfrac{4}{p\alpha}\cdot \tfrac{1}{f(r)}\cdot \tfrac{1}{||u||_{L_t^\infty H_\mu^1}^2}.\end{align} By ($\ref{fbound1}$), we know $\lambda(r)>\frac{1}{\pi}$ and $0<\theta(r)<1$ uniformly in $(2,\infty)$. 

Consider the continuous function $\phi:(2,\infty)\to \R$ given by \begin{align*}\phi(r)&=\tfrac{1}{r}-\left[\tfrac{\theta(r)}{4}+\tfrac{1-\theta(r)}{2}\right]=\tfrac{r-2}{2r}\big[\tfrac{1}{p\alpha}\cdot \tfrac{1}{f(r)}\cdot \tfrac{1}{||u||_{L_t^\infty H_{\mu}^1}^2}-1\big].\end{align*}
 By $\eqref{critineq1}$ and $\eqref{flimit1}$,  \begin{align*}&\lim_{r\to 2^{+}}\frac{1}{p\alpha}\cdot \frac{1}{f(r)}\cdot \frac{1}{||u||_{L_t^\infty H_{\mu}^1}^2}-1=\frac{1}{p \alpha ||u||_{L_t^\infty H_{\mu}^1}^2}-1>0\ \ \text{ and }\\ &\lim_{r\to\infty}\frac{1}{p\alpha}\cdot \frac{1}{f(r)}\cdot \frac{1}{||u||_{L_t^\infty H_{\mu}^1}^2}-1=-1<0.\end{align*}  Since $\frac{r-2}{2r}>0$ for $r>2$, it follows by the intermediate value theorem that there exists some $r_0\in (2,\infty)$ so that $\phi(r_0)=0$. For this value of $r_0$, we have that 
 \begin{align*}
 \frac{\theta(r_0)}{4}=\frac{\theta(r_0)}{4}+\frac{1-\theta(r_0)}{\infty}\qquad \text{ and }\qquad\frac{1}{r_0}=\frac{\theta(r_0)}{4}+\frac{1-\theta(r_0)}{2}.\end{align*}By interpolation we deduce that \begin{align}\label{interp}||u||_{L_t^{\frac{4}{\theta(r_0)}}H_{x}^{1,r_0}}\le ||u||_{L_t^4H_{x}^{1,4}}^{\theta(r_0)}||u||_{L_t^\infty H_{x}^{1}}^{1-\theta(r_0)}.\end{align}
 
 Now, let \begin{align*}K&=2(1+M)C_{\lambda(r_0)}\qtq{and}I=\{t\in \R: ||u(t)||_{L_x^\infty(\R^2)}>K\},\end{align*} where $C_{\lambda(r_0)}$ is the constant appearing in $\eqref{bwg}$. Notice that if $t\in I$ we have 
\[C_{\lambda(r_0)}||u||_{L_t^\infty H_{\mu}^1}^2 \le C_{\lambda(r_0)}(1+M)< \tfrac{1}{2}||u(t)||_{L_x^\infty}\le \tfrac{c_{r_0}}{2}||u(t)||_{H_x^{1,r_0}},\] from which it follows that \begin{align*}C_{\lambda(r_0)}<  \frac{c_{r_0}||u(t)||_{H_x^{1,r_0}}}{2||u||_{L_t^\infty H_{\mu}^1}^2}.\end{align*} Thus \begin{align} \label{maxbound}C_{\lambda}(r_0)+\frac{(8/\mu)^{1-\frac{2}{r_0}}c_{r_0}||u(t)||_{H_x^{1,r_0}}}{||u||_{L_t^\infty H_{\mu}^1}}\lesssim_{M,\mu} ||u(t)||_{H_x^{1,r_0}}.\end{align} 

With our parameters now well understood, the proof is rather straightforward. By Proposition $\ref{SBWG}$, $\eqref{thetadef}$, and $\eqref{maxbound}$,
\begin{align*}
\exp[4\pi \alpha ||u(t)||_{L_x^\infty(\R^2)}^2]
&\le \exp\left[4\pi \lambda(r_0) \alpha ||u(t)||_{H_{\mu}^1}^2\log \left[C_{\lambda}(r_0)+\frac{(8/\mu)^{1-\frac{2}{r_0}}c_{r_0}||u(t)||_{H_x^{1,r_0}}}{||u(t)||_{H_{\mu}^1}}\right]\right]\\ 
&\le \exp\left[4\pi \lambda(r_0) \alpha ||u||_{L_t^\infty H_{\mu}^1}^2\log \left[C_{\lambda}(r_0)+\frac{(8/\mu)^{1-\frac{2}{r_0}}c_{r_0}||u(t)||_{H_x^{1,r_0}}}{||u||_{L_t^\infty H_{\mu}^1}}\right]\right]\\
&\lesssim_{M,\mu} ||u(t)||_{H_x^{1,r_0}}^{4\pi \alpha \lambda(r_0)||u||_{L_t^\infty H_{\mu}^1}^2}\\ 
&=||u(t)||_{H_{x}^{1,r_0}}^{\frac{4}{p\theta(r_0)}}.\end{align*} Note that the second inequality follows from the fact that if $a>1$ and $b>0$ then the function \[x\mapsto x^2\log(a+\tfrac{b}{x})\] is an increasing function when $x>0$. Integrating in time we obtain \begin{align*}||\exp(4\pi \alpha |u|^2)||_{L_t^pL_x^\infty(\R\times \R^2)}\lesssim_{M,\mu} ||u||_{L_t^\frac{4}{\theta(r_0)}H_{x}^{1,r_0}}^{\frac{4}{p\theta(r_0)}}\lesssim_{M,\mu} ||u||_{L_t^4H_{x}^{1,4}}^{\frac{4}{p}}\end{align*}as desired.
\end{proof}
\end{lemma}
We divide the proof of Theorem $\ref{GMTS}$ into two propositions. 

\begin{proposition}\label{GMTSP1} Let $u:\R\times \R^2\to \C$ satisfy \begin{align}\label{cond12}||u(t)||_{\dot{H}_x^1(\R^2)}\le 1\qtq{and}||u(t)||_{L_x^2(\R^2)}^2=M,\end{align} for every $t\in \R$. If $s\ge 4$ and $1\le p,q\le \infty$ are such that \begin{align}\label{cond22}\frac{q'}{p}> 1,\end{align}then \begin{align}\label{GMTSC2}|||u|^se^{4\pi |u|^2}||_{L_t^pL_x^q(\R\times \R^2)}\lesssim_{M} ||u||_{L_t^4H_{x}^{1,4}(\R\times \R^2)}^{\frac{4}{p}}.\end{align}
If $||u||_{L_t^\infty \dot{H}_x^1}<1$, then we may replace condition $\eqref{cond22}$ with the condition\begin{align}\frac{q'}{p}\ge 1.
\end{align}

\begin{proof}Let $K>1$ be a large parameter whose value will be chosen later, and define \[I_{K}:=\{t\in \R: ||u(t)||_{L_x^\infty(\R^2)}>K\}.\] 

If $t\not\in I_{K}$, then since $s\ge 4$ we have that \begin{align}\label{trivbd}|u(t,x)|^se^{4\pi |u(t,x)|^2}\lesssim_{K}|u(t,x)|^4.\end{align} By H\"older and Sobolev Embedding it follows that \[ ||u^4(t)||_{L_x^q}\le ||u(t)||_{L_x^4}^{\frac{4}{q}}||u(t)||_{L_x^\infty}^\frac{4}{q'}\lesssim ||u(t)||_{L_x^4}^{\frac{4}{q}}||u(t)||_{H_x^{1,4}}^{\frac{4}{q'}}.\] As $||u(t)||_{L_x^4}\le ||u(t)||_{H_x^{1,4}}$ and $q'\ge p$ we see that
\begin{align*}||u(t)||_{L_x^4}^{\frac{4}{q}}||u(t)||_{H_x^{1,4}}^{\frac{4}{q'}}\le ||u(t)||_{L_x^4}^{\frac{4}{p'}}||u(t)||_{H_{x}^{1,4}}^{\frac{4}{q}-\frac{4}{p'}}||u(t)||_{H_{x}^{1,4}}^{\frac{4}{q'}}= ||u(t)||_{L_x^4}^\frac{4}{p'}||u(t)||_{H_{x}^{1,4}}^{\frac{4}{p}}.\end{align*}Integrating in time, we obtain that \begin{align}\label{lowlowest}||u^4(t)||_{L_t^pL_x^q((\R\setminus I_{K})\times \R^2)}\lesssim_{K} ||u||_{L_t^\infty L_x^4}^{\frac{4}{p'}}||u||_{L_t^4H_{x}^{1,4}}^{\frac{4}{p}}\lesssim_{M,K}||u||_{L_t^4H_{x}^{1,4}}^\frac{4}{p}.\end{align} In light of \eqref{trivbd}, we deduce that 
\begin{align}\label{lowest}|||u|^se^{4\pi |u|^2}||_{L_t^pL_x^q((\R\setminus I_{K})\times \R^2)}^p\lesssim_{M,K}||u||_{L_t^4H_{x}^{1,4}}^{\frac{4}{p}}.\end{align} provided that $q'\ge p$. 

We now turn to estimating $u$ on $I_{K}$. Let $\alpha\in (0,1)$ be close to $1$. In light of Proposition \ref{MT4} we have that for each $t\in I_{K}$ that \begin{align*}
|||u(t)|^se^{4\pi |u(t)|^2}||_{L_x^q}^q&=\int_{\R^2}|u(t,x)|^{sq}e^{4\pi q|u(t,x)|^2}\ dx\\ &=e^{4\pi (q-\alpha)||u(t)||_{L_x^\infty(\R^2)}^2}\int_{\R^2}|u(t,x)|^{sq}e^{4\pi \alpha|u(t,x)|^2}\ dx\\ &\lesssim e^{4\pi (q-\alpha)||u(t)||_{L_x^\infty(\R^2)}^2} ||u(t)||_{L_x^{sq}(\R^2)}^{sq}.
\end{align*}Taking $q^{\text{th}}$ roots, we arrive at \[|| |u|^se^{4\pi |u|^2}||_{L_t^pL_x^q(I_{K}\times \R^2)}\lesssim e^{4\pi \frac{q-\alpha}{q}||u(t)||_{L_x^\infty(\R^2)}^2}||u||_{L_x^{sq}(\R^2)}^{s}.\] Integrating in time and interpolating we deduce that 
\begin{align*}|||u|^se^{4\pi |u|^2}||_{L_t^pL_x^q(I_{K}\times \R^2)}&\lesssim ||u||_{L_t^\infty L_x^{sq}(I_{K}\times \R^2)}^s||e^{4\pi \frac{q-\alpha}{q}|u|^2}||_{L_t^pL_x^\infty(I_{K}\times \R^2)}\\&\lesssim_{M} ||e^{4\pi \frac{q-\alpha}{q}|u|^2}||_{L_t^pL_x^\infty(I_{K}\times \R^2)}.\end{align*} If $\frac{q'}{p}>||u||_{L_t^\infty \dot{H}_x^1}$, then \[\lim_{\alpha\to 1^{-}}\frac{1}{p\frac{q-\alpha}{q}}=\frac{q'}{p}>||u||_{L_t^\infty\dot{H}_x^1(\R\times \R^2)},\] and so we may apply Lemma \ref{LEST} to obtain that for $K$ sufficiently large, \begin{align}\label{highest}|||u|^se^{4\pi |u|^2}||_{L_t^pL_x^q(I_{K}\times \R^2)}\lesssim_{M,K}||u||_{L_t^4H_{x}^{1,4}(\R\times \R^2)}^\frac{4}{p}.\end{align} 
Combining \eqref{lowest} and \eqref{highest} we deduce that \[|||u|^se^{4\pi |u|^2}||_{L_t^pL_x^q}\lesssim_{M}||u||_{L_t^4H_{x}^{1,4}(\R\times \R^2)}^\frac{4}{p}\] when $\frac{q'}{p}\ge 1$ and $\frac{q'}{p}>||u(t)||_{L_t^\infty \dot{H}_x^1}$, as desired. 
\end{proof} 
\end{proposition}
A few remarks on the previous theorem are in order.

\begin{remark}
The condition that \[\frac{q'}{p}> 1,\] is equivalent to condition that \[\tfrac{1}{p}+\tfrac{1}{q}>1.\]
\end{remark}

\begin{remark}\label{difftech}Let us pause to illustrate a slightly different technique in estimating the purely exponential nonlinearity \[||e^{4\pi |u|^2}-1||_{L_t^pL_x^q(I_{K}\times \R^2)}.\] Rather than employ Proposition $\ref{MT4}$, we instead estimate using Lemma $\ref{L1EST}$.   Indeed, by interpolation and Lemma $\ref{L1EST}$ we see that for any $\eta>0$,\begin{align*}
||e^{4\pi |u|^2}-1||_{L_x^q}&\le ||e^{4\pi|u|^2}-1||_{L_x^1}^{\frac{1}{q}}||e^{4\pi |u|^2}-1||_{L_x^\infty}^{\frac{1}{q'}}\\ &\le C(M)(1+||u||_{L_x^\infty}^2)^{\frac{1}{q}}||e^{4\pi |u|^2}-1||_{L_x^\infty}^{\frac{1}{q'}}\\ &\le C(M,\eta)||e^{4\pi \frac{(1+\eta)}{q'}|u|^2}||_{L_x^\infty}.
\end{align*} If $\frac{q'}{p}>||u||_{L_t^\infty \dot{H}_{x}^{1}}$, we may choose $\eta$ sufficiently small as to guarantee that \[\frac{1}{\frac{(1+\eta)}{q'}\cdot p}>||u||_{L_t^\infty\dot{H}_x^1}.\] Applying Lemma $\ref{LEST}$ with $\alpha=\frac{1+\eta}{q'}$, there exists some $K>0$ so that 
\begin{align}\label{MTMINUS}||e^{4\pi |u|^2}-1||_{L_t^pL_x^q(I_{K}\times \R^2)}\le C(M,\eta)||u||_{L_t^4H_{x}^{1,4}}^{\frac{4}{p}}.\end{align}
\end{remark}

We observe that Theorem $\ref{GMTS}$ fails at the end-point.

\begin{proposition}If $1\le p,q\le \infty$ are such that \[|||u|^se^{4\pi |u|^2}||_{L_t^pL_x^q(\R\times \R^2)}\lesssim ||u||_{L_t^4H_{x}^{1,4}}^{\frac{4}{p}},\] for every spacetime function $u:\R\times \R^2\to \C$ satisfying\begin{align}||u(t)||_{\dot{H}_x^1(\R^2)}\le 1\qtq{and}||u(t)||_{L_x^2(\R^2)}^2=M,\end{align} for every $t\in \R$, then \[\frac{q'}{p}>1.\]
\begin{proof}As in \cite{MR2929605}, for $N\ge 3$ we consider the sequence of functions \[v_N(t,x)=\sqrt{\frac{2\pi}{\log N}}\frac{1}{(2\pi)^2}\int_{1<|\xi|<e^{-1}N}|\xi|^{-2}e^{-it|\xi|^2+i\xi\cdot x}\ d\xi.\] Note that $v_N(t,x)=e^{it\Delta}v_N(0,x)$, where \[v_N(0,x)=\sqrt{\frac{2\pi}{\log(N)}}\frac{1}{(2\pi)^2}\int_{1<|\xi|<e^{-1}N}|\xi|^{-2}e^{i\xi\cdot x}\ d\xi.\]  By conservation of mass and energy it follows that \begin{align}||v_N||_{L_t^\infty L_x^2}^2&=||v_N(0)||_{L_x^2}^2<\frac{1}{2\log N},\qquad \text{and } \label{logbd1}\\ ||v_N||_{L_{t}^{\infty}\dot{H}_{x}^{1}}^2&= ||\grad v_N(0)||_{L_x^2}^2=\frac{\log(N)-1}{\log(N)}\label{logbd2}.\end{align} On the regions $T_N:=\{(t,x): t\sim \ep N^{-2}\text{ and }|x|\sim \ep N^{-1}\}$, we may estimate \[\text{Re}(v_N(t,x))\ge \sqrt{\tfrac{\log(N)}{2\pi}}-\tfrac{1+\ep^2}{\sqrt{2\pi \log N}}\text{ and }|v_N|^se^{4\pi |v_N|^2}\gtrsim N^2(\log(N))^\frac{s}{2},\] when $\ep>0$ is sufficiently small and $N$ is sufficiently large. But then we see that  
\[|||v_N|^se^{4\pi |v_N|^2}||_{L_t^pL_x^q}\ge |||v_N|^se^{4\pi |v_N|^2}||_{L_t^pL_x^q(T_N)}\gtrsim N^{2(1-\frac{1}{p}-\frac{1}{q})}(\log N)^\frac{s}{2}.\] By the Strichartz inequality, $\eqref{logbd1}$, and $\eqref{logbd2}$ we obtain that \[(\log N)^\frac{s}{2}N^{2(1-\frac{1}{p}-\frac{1}{q})}\lesssim_{M} ||v_N||_{L_t^4H_{x}^{1,4}}^{\frac{4}{p}}\lesssim_{M} ||v_N(0)||_{H_x^1}^{\frac{4}{p}}\lesssim 1,\] or, equivalently, that $\frac{1}{p}+\frac{1}{q}> 1.$
\end{proof}\end{proposition}

Throughout the perturbative theory, we will need the following estimate. The presence of the Morawetz norm $L_t^4L_x^8$ in the statement means that we cannot solely rely on the Moser--Trudinger--Strichartz inequality, since this will only produce an estimate involving the $L_t^4H_{x}^{1,4}$ norm. The difference will be a superficial one as we merely repeat the proof of Moser--Trudinger--Strichartz with a timely H\"older's inequality. 

\begin{corollary}[\cite{MR2929605}] \label{subcritmora}For each $H\in (0,1)$ there exists $\delta=\delta(H)\in (0,1/3)$ so that for any strong solution $u$ of $\eqref{PDE2}$ on $I\times \R^2$ with $H_2(u)\le H$ we have \[||F_2(u)||_{N(I)}\lesssim C(H,M)||u||_{L_t^4 L_x^8(I\times \R^2)}^{4\delta}||u||_{L_t^4H_x^{1,4}(I\times \R^2)}^{2},\] where $M=||u(0)||_{L_x^2}^{2}$. 
\begin{proof}Let $0<\delta<\frac{1}{3}$, $K\gg 1$, and, as before, $I_{K}=\{t: ||u(t)||_{L_x^\infty}>K\}$.  

We note that \[||\grad|^k F_2(u)|\lesssim ||\grad|^ku|\cdot |u|^2 \cdot |u|^2e^{4\pi |u|^2},\] for $k\in \{0,1\}$. 

We first estimate the nonlinearity when $t\not\in I_{K}$. In this case, we have that \begin{align*}||\grad|^kF_2(u(t,x))|&\lesssim_{K} ||\grad|^k u(t,x)||u(t,x)|^4.\end{align*} So on the set $(\R\setminus I_{K})\times \R^2$ we have 
\begin{align*}
|||\grad|^kF_2(u)||_{L_t^{\frac{2}{1+2\delta}}L_x^{\frac{1}{1-\delta}}}&\lesssim_{K}|| |\grad|^ku||_{L_t^{\frac{2}{\delta}}L_x^{\frac{2}{1-\delta}}} || u^2||_{L_t^{\frac{1}{\delta}}L_x^{\frac{2}{\delta}}} ||u^2||_{L_t^{\frac{2}{1-\delta}}L_x^{\frac{2}{1-2\delta}}}\\ &\lesssim_{K}|| |\grad|^ku||_{L_t^{\frac{2}{\delta}}L_x^{\frac{2}{1-\delta}}} || u||_{L_t^{\frac{2}{\delta}}L_x^{\frac{4}{\delta}}}^2 ||u^4||_{L_t^{\frac{1}{1-\delta}}L_x^{\frac{1}{1-2\delta}}}^\frac{1}{2}.
\end{align*}

By interpolation, \begin{align}\label{interp}|||\grad|^ku||_{L_t^{\frac{2}{\delta}}L_x^\frac{2}{1-\delta}}\lesssim |||\grad|^ku||_{L_t^\infty L_x^2}^{1-2\delta}|||\grad|^ku||_{L_t^4 L_x^4}^{2\delta}\lesssim ||u||_{L_t^4H_{x}^{1,4}}^{2\delta}.\end{align}

We have seen from \eqref{lowlowest} in Proposition \ref{GMTSP1} that, since $(1-\delta)+(1-2\delta)>1$,
\[||u^4||_{L_t^{\frac{1}{1-\delta}}L_x^{\frac{1}{1-2\delta}}}^\frac{1}{2}\lesssim (||u||_{L_t^4H_{x}^{1,4}}^{4(1-\delta)})^{\frac{1}{2}}=||u||_{L_t^4H_{x}^{1,4}}^{2-2\delta}.\]

By the Gagliardo-Nirenberg interpolation inequality, we know that \[||u(t)||_{L_x^{\frac{4}{\delta}}}\lesssim ||\grad u(t)||_{L_x^2}^{1-2\delta}||u(t)||_{L_x^{8}}^{2\delta}.\] So, integrating in time, we see that \begin{align}\label{gn}||u||_{L_t^{\frac{2}{\delta}}L_x^{\frac{4}{\delta}}}\lesssim ||\grad u||_{L_t^\infty L_x^2}^{1-2\delta}||u||_{L_t^4L_x^8}^{2\delta}.\end{align} Putting this together, we see that on $(\R\setminus I_{K})\times \R^2$ that \[|||\grad|^kF_2(u)||_{L_t^{\frac{2}{1+2\delta}}L_x^{\frac{1}{1-\delta}}}\lesssim ||u||_{L_t^4L_x^8}^{4\delta}||u||_{L_t^4H_x^{1,4}}^{2},\] as desired.

We now turn to estimate the nonlinearity when $t\in I_{K}$. By ($\ref{MTMINUS}$) in Remark \ref{difftech}, for each $1\le p,q\le \infty$ with $\frac{q'}{p}>||u||_{L_t^\infty \dot{H}_x^1}$ we may find a $K>0$ so that  \[||e^{4\pi |u|^2}-1||_{L_t^p L_x^q(I_{K}\times \R^2)}\lesssim ||u||_{L_t^4H_{x}^{1,4}}^\frac{4}{p}.\] So by H\"older we find that on the spacetime region $I_{K}\times \R^2$ \[|||\grad|^k F_2(u)||_{L_t^\frac{2}{1+2\delta}L_x^{\frac{1}{1-\delta}}}\lesssim |||\grad|^ku||_{L_t^\frac{2}{\delta} L_x^{\frac{2}{1-\delta}}}||u||_{L_t^\frac{2}{\delta}L_x^\frac{4}{\delta}}^2||e^{4\pi |u|^2}-1||_{L_t^\frac{2}{1-\delta} L_x^{\frac{2}{1-2\delta}}}.\] 

As \[\lim_{\delta\to 0}\tfrac{\left(\frac{2}{1-2\delta}\right)'}{\frac{2}{1-\delta}}=1>||u||_{L_t^\infty \dot{H}_x^1},\] for sufficiently small $\delta>0$ we know that  \[||e^{4\pi |u|^2}-1||_{L_t^{\frac{2}{1-\delta}}L_x^{\frac{2}{1-2\delta}}(I_{K}\times \R^2)}\lesssim ||u||_{L_t^4H_{x}^{1,4}}^{2(1-\delta)}.\] 

Estimating the other terms similarly, as in \eqref{gn} and \eqref{interp}, and putting everything together, we see that 
\begin{align*}|||\grad|^k F_2(u)||_{L_t^\frac{2}{1+2\delta}L_x^{1,\frac{1}{1-\delta}}}\lesssim C(H,M)||u||_{L_t^4L_x^8}^{4\delta}||u||_{L_t^4H_{x}^{1,4}}^2,\end{align*} as desired. 
\end{proof}
\end{corollary}

 \subsection{Local Moser--Trudinger--Strichartz}  Employing similar techniques and using H\"older's inequality in time, we may derive a local-in-time Moser--Trudinger--Strichartz inequality. As in the global case, an estimate on the $L_t^\infty L_x^1$ norm will follow from the Moser--Trudinger inequality, an estimate on the $L_t^pL_x^\infty$ norm will follow from the Br\'ezis--Wainger--Gallou\"et inequality, and a general spacetime estimate will follow from interpolating in-between. We give the statement in the energy-subcritical setting where it is of most use. As a corollary, we will derive dual Strichartz estimates on the nonlinearity which will streamline the well-posedness theory in the next section. 
 
 For the sake of exposition, we introduce the following function space: on a slab $I\times \R^2$, we define $X(I)$ to be the closure of test functions under the norm \begin{align*}
||u||_{X(I)}&:=||u||_{L_t^4H_x^{1,4}(I\times \R^2)}+||u||_{L_t^\infty H_x^{1}(I\times \R^2)}.
\end{align*} In the special case $I=[-T,T]$ we simply write $X_T=X([-T,T])$. 

\begin{lemma} \label{LLMTS}Suppose $u\in L_t^\infty H_{x}^{1}([-T,T]\times \R^2)$ satisfies \[||u||_{L_t^\infty\dot{H}_x^1}\le A<1\qtq{and}||u||_{L_{T}^\infty L_x^2(\R^2)}^2=M.\] If $\beta\in (0,1]$ and $1\le p<\infty$ satisfy \[A^2<\frac{1}{p\beta},\] then there exists \[0<\gamma=\gamma(A,M,p,\beta)<\min\{4\beta,\tfrac{4}{p}\}\] and $\ep_0=\ep_0(A,M,\beta,p)>0$ so that \begin{align}\label{LMTSeq}||e^{4\pi (1+\ep)\beta |u|^2}-1||_{L_t^{p}L_x^\infty([-T,T]\times \R^2)}\lesssim_{A,M} T^{\frac{4-p\gamma}{4p}}(T^{\frac{1}{4}}+||u||_{L_T^4H_{x}^{1,4}})^{\gamma}.\end{align}for all $0<\ep<\ep_0$. Moreover, if $p\beta\le 1$, then $0<\gamma=\gamma(A,M,\beta)<4\beta$ and $\ep_0=\ep_0(A,M)$. 
\begin{proof}We begin our proof, as we did in Lemma $\ref{LEST}$, by choosing some parameters first. Choose $1\ge \mu=\mu(A,M,\beta,p)>0$, $\ep_0=\ep_0(A,M,\beta,p)>0$, and $\alpha=\alpha(A,M,\beta,p)>1$ so that \begin{align}\label{paramparam} \alpha (1+\ep_0)^2(A^2+\mu M)<\min\{1,\tfrac{1}{p\beta}\}.\end{align} 

Let $0<\ep<\ep_0$. By Proposition $\ref{SBWG}$ we know that for some $C_{\alpha}>1$ that \[\exp(4\pi (1+\ep)\beta ||u(t)||_{L_x^\infty}^2)\le \exp\left(4\alpha \beta (1+\ep)||u(t)||_{H_{\mu}^1}^2\log \left(C_{\alpha}+\tfrac{c_{4}(8/\mu)^{\frac{1}{2}}||u(t)||_{H_{x}^{1,4}}}{||u(t)||_{H_{\mu}^1}}\right)\right)\] We recall that for $a>1$ and $b>0$ the function \[x\mapsto x^2\log(a+\tfrac{b}{x})\] is increasing for $x>0$. Since $||u(t)||_{H_{\mu}^1}\le \sqrt{A^2+\mu M}$, it follows that\begin{align} \exp(4\pi (1+\ep)\beta ||u(t)||_{L_x^\infty}^2)&\lesssim_{A} \left(1+\tfrac{||u(t)||_{H_{x}^{1,4}}}{\sqrt{A^2+\mu M}}\right)^{4\alpha \beta(1+\ep)(A^2+\mu M)}\end{align}

Let $\gamma:=4\alpha \beta(1+\ep)(A^2+\mu M)$. By $\eqref{paramparam}$ we know that $0<\gamma<\min\{4\beta,4/p\}$ and, in particular, that $p\gamma<4$. 

Integrating in time, by H\"older's inequality we see that
\begin{align*}
||e^{4\pi (1+\ep)\beta |u|^2}||_{L_t^pL_x^\infty}&\lesssim_{A,M}||1+\tfrac{||u(t)||_{H_{x}^{1,4}}}{A}||_{L_t^{p\gamma}}^{\gamma}\\ &\lesssim_{A,M}T^{\frac{4-p\gamma}{4p}}||1+\tfrac{||u(t)||_{H_{x}^{1,4}}}{A}||_{L_T^{4}}^{\gamma}\\ &\lesssim_{A,M} T^{\frac{4-p\gamma}{4p}}(T^{\frac{1}{4}}+||u||_{L_T^4H_{x}^{1,4}})^{\gamma},
\end{align*} as desired. 
\end{proof}
\end{lemma}

Combining Lemma $\ref{LLMTS}$ with Proposition $\ref{MT1}$ we obtain the following corollary by interpolation. 
\begin{theorem}[Local Moser--Trudinger--Strichartz]\label{LMTS}Suppose $u\in L_t^\infty H_{x}^{1}([-T,T]\times \R^2)$ satisfies \[||u||_{L_t^\infty\dot{H}_x^1}\le A<1\qtq{and}||u||_{L_{T}^\infty L_x^2(\R^2)}^2=M.\] If $1\le p\le \infty$ and $1<q\le \infty$ satisfy \[\frac{q'}{p}> A^2\] then there exists $0<\gamma=\gamma(A,M,p,q)<\min\{\frac{4}{q'},\frac{4}{p}\}$ and $\ep_0=\ep_0(A,p,q)>0$ so that \begin{align}\label{LMTSCon} ||e^{4\pi (1+\ep)|u|^2}-1||_{L_T^pL_x^q}\lesssim ||u||_{L_t^\infty L_{x}^2}^{\frac{2}{q}}T^{\frac{4-p\gamma}{4p}}(T^{\frac{1}{4}}+||u||_{L_T^4H_{x}^{1,4}})^{\gamma}.\end{align} for every $0<\ep<\ep_0$. Moreover, if $\frac{q'}{p}\ge 1$, then $0<\gamma=\gamma(A,M,q)<\frac{4}{q'}$ and $\ep_0=\ep_0(A,M)$. 
\begin{proof}Taking $\beta=\tfrac{1}{q'}$ in Lemma $\ref{LLMTS}$ we obtain an $\ep_0=\ep_0(A,M)>0$ and  a $0<\gamma<\tfrac{4}{q'}$ so that $\eqref{LMTSeq}$ holds for all $0<\ep<\ep_0$ and $p\in [1,q']$. By interpolation and Proposition $\ref{MT1}$ we see that 
\begin{align*}
||e^{4\pi |u(t)|^2}-1||_{L_x^q}&\le ||e^{4\pi |u(t)|^2}-1||_{L_x^1}^{\frac{1}{q}}||e^{4\pi |u(t)|^2}-1||_{L_{x}^{\infty}}^{\frac{1}{q'}}\\ &\lesssim ||u(t)||_{L_x^2}^\frac{2}{q}||e^{\frac{4\pi }{q'}|u(t)|^2}-1||_{L_x^\infty}.
\end{align*}Integrating in time we see by Lemma $\ref{LLMTS}$ that \begin{align*}
||e^{4\pi |u|^2}-1||_{L_T^pL_x^q}&\lesssim ||u||_{L_t^\infty L_{x}^2}^{\frac{2}{q}}||e^{\frac{4\pi}{q'}|u|^2}-1||_{L_x^pL_{x}^{\infty}}\\ &\lesssim_{A,M} ||u||_{L_t^\infty L_{x}^2}^{\frac{2}{q}}T^{\frac{4-p\gamma}{4p}}(T^{\frac{1}{4}}+||u||_{L_T^4H_{x}^{1,4}})^{\gamma},
\end{align*} as desired. 

\end{proof}
\end{theorem}
\begin{corollary} \label{NLH1}Suppose $u_1,u_2\in L_t^\infty H_{x}^{1}([-T,T]\times \R^2)$ satisfy \[||u_i||_{L_t^\infty\dot{H}_x^1}\le A<1\qtq{and}||u_i||_{L_{T}^\infty L_x^2(\R^2)}^2=M.\] for each $i\in \{1,2\}$, then there is some $0<\gamma=\gamma(A)<3$ so that \begin{align*}||F_i(u_1)-F_i(u_2)||_{L_T^1L_x^2}\lesssim_{A,M}  ||u_1-u_2||_{X_T}\sum_{j=1,2} ||u_j||_{X_T}^{\frac{1}{2}}T^{\frac{3-\gamma}{4}}(T^{\frac{1}{4}}+||u_j||_{X_T})^{\gamma}.\end{align*}
\begin{proof} Our pointwise estimates $\eqref{F1diff}$ and $\eqref{F2diff}$ imply that for any $\ep>0$ that 
\begin{align*}||F_i(u_1)-F_i(u_2)||_{L_T^1 L_x^2}&\lesssim ||(u_1-u_2)\sum_{j=1,2}e^{4\pi (1+\ep)|u_j|^2}-1||_{L_T^1 L_x^2}\\ &\lesssim ||u_1-u_2||_{L_T^4L_x^4}\sum_{j=1,2}||e^{4\pi(1+\ep)|u_j|^2}-1||_{L_T^\frac{4}{3}L_x^4}.\end{align*}Applying Theorem \ref{LMTS} with $q'=p=\frac{4}{3}$ we obtain an $\ep_0=\ep_0(A)>0$ and a $0<\gamma<3$ so that \eqref{LMTSCon} holds for all $0<\ep<\ep_{0}$. This grants us that \[||F_i(u_1)-F_i(u_2)||_{L_T^1 L_x^2}\lesssim ||u_1-u_2||_{L_T^4L_x^4}\sum_{j=1,2} ||u_j||_{L_t^\infty L_{x}^2}^{\frac{1}{2}}T^{\frac{3-\gamma}{4}}(T^{\frac{1}{4}}+||u_j||_{L_T^4H_{x}^{1,4}})^{\gamma},\] as desired. 
\end{proof}\end{corollary}

\begin{corollary}\label{NLH2}Suppose $u_1,u_2\in L_t^\infty H_{x}^{1}([-T,T]\times \R^2)$ satisfy \[||u_i||_{L_t^\infty\dot{H}_x^1}\le A<1\qtq{and}||u_i||_{L_{T}^\infty L_x^2(\R^2)}^2=M\] for each $i\in \{1,2\}$. Then there is some $0<\gamma=\gamma(A,M)<3$ and $\delta=\delta(A)>0$ so that \begin{align*}||F_i(u_1)-F_i(u_2)||_{L_T^1\dot{H}_x^1}&\lesssim_{A}  T^{\frac{3}{4}}||u_1-u_2||_{X_T}\sum_{j=1,2}(||u_j||_{X_{T}}+||u_j||_{X_{T}}^2)\\ &+T^{\frac{3-\gamma}{4}}||u_1-u_2||_{X_T}\sum_{j=1,2}(||u_j||_{X_T}^{\frac{1-4\delta}{2}}+||u_j||_{X_T}^{\frac{3-4\delta}{2}})(T^{\frac{1}{4}}+||u_j||_{X_{T}})^{\gamma}.\end{align*}
\begin{proof} Write 
\begin{align*}\grad F_i(u_1)-\grad F_i(u_2)&= [(\partial_zF_i)(u_2)](\grad u_1-\grad u_2)+\grad u_1[(\partial_zF_i)(u_1)-(\partial_zF_i)(u_2)]\\ &+[(\partial_{\bar{z}}F_i)(u_2)]\overline{(\grad u_1-\grad u_2)}+\overline{\grad u_1}[(\partial_{\bar{z}}F_i)(u_1)-(\partial_{\bar{z}}F_i)(u_2)\\ &=(I)+(II)+(III)+(IV).\end{align*} 

We first bound $(I)+(III)$. Towards this end, note that by H\"older and our pointwise bound \eqref{F1derivdiff2} we obtain that
\begin{align*}
||(I)+(III)||_{L_T^1L_x^2}&\lesssim ||\grad u_1-\grad u_2||_{L_T^4L_x^4}(||(\partial_z F_i)(u_2)||_{L_T^\frac{4}{3}L_x^4}+||(\partial_{\bar{z}} F_i)(u_2)||_{L_T^\frac{4}{3}L_x^4})\\ &\lesssim ||\grad u_1-\grad u_2||_{L_T^4L_x^4}|||u_2|(|u_2|+(e^{4\pi (1+\ep)|u_2|^2}-1))||_{L_T^\frac{4}{3}L_{x}^{4}}\\ &\lesssim ||\grad u_1-\grad u_2||_{L_T^4L_x^4}(||u_2||_{L_T^{\frac{8}{3}}L_{x}^{8}}^2+||u_2(e^{4\pi (1+\ep)|u_2|^2}-1)||_{L_T^\frac{4}{3}L_{x}^{4}})\\ &\lesssim ||\grad u_1-\grad u_2||_{L_T^4L_{x}^{4}}(T^{\frac{3}{4}}||u_{2}||_{L_T^{\infty}L_x^8}^2+||u_2(e^{4\pi (1+\ep)|u_2|^2}-1)||_{L_T^\frac{4}{3}L_{x}^{4}}).
\end{align*}

To control the spacetime norm involving the exponential term, we argue as follows: for each $\delta>0$ we have by H\"older's inequality that
\begin{align*}
||u_2(e^{4\pi (1+\ep)|u_2|^2}-1)||_{L_T^\frac{4}{3}L_x^4}\le ||u_2||_{L_T^\infty L_x^{\frac{1}{\delta}}}||e^{4\pi (1+\ep)|u_2|^2}-1||_{L_T^\frac{4}{3} L_x^\frac{4}{1-4\delta}}.\end{align*}Choose $\delta=\delta(A)$ small enough to guarantee that \[A^2<\tfrac{(\tfrac{4}{1-4\delta})'}{\tfrac{4}{3}}=\tfrac{3}{3+4\delta}.\] Applying Theorem \ref{LMTS} with $p=\frac{4}{3}$ and $q=\frac{4}{1-4\delta}$ we obtain an $\ep_0=\ep_0(A)>0$ and a $0<\gamma<3$ so that \eqref{LMTSCon} holds for all $0<\ep<\ep_{0}$. This grants us that \[||e^{4\pi (1+\ep)|u_2|^2}-1||_{L_T^{\frac{4}{3}}L_x^{\frac{4}{1-4\delta}}}\lesssim_{A,M} ||u_2||_{L_t^\infty L_{x}^2}^{\frac{1-4\delta}{2}}T^{\frac{3-\gamma}{4}}(T^{\frac{1}{4}}+||u_2||_{L_T^4H_{x}^{1,4}})^{\gamma}.\] 

Putting these together and employing Sobolev Embedding \eqref{sob} we see that \begin{align*}
||(I)+(III)||_{L_T^1L_x^2}\lesssim_{A,M}||u_1-u_2||_{X_{T}}\left[T^{\frac{3}{4}}||u_2||_{X_{T}}^2+||u_2||_{X_T}^{\frac{3-4\delta}{2}}T^{\frac{3-\gamma}{4}}(T^{\frac{1}{4}}+||u_2||_{X_{T}}\right]^{\gamma}.
\end{align*}

We now bound $(II)+(IV)$. Towards this end, we note again that by H\"older and applying our pointwise bound \eqref{F1derivdiff2} with $z_2=0$ we obtain that 
\begin{align*}
||(II)+(IV)||_{L_T^1L_x^2} &\lesssim ||\grad u_1||_{L_T^4L_x^4}|||u_1-u_2|\sum_{j=1,2}(e^{4\pi (1+\ep)|u_j|^2}-1+|u_j|)||_{L_T^\frac{4}{3}L_x^4}\\ &\lesssim ||\grad u_1||_{L_T^4L_x^4}\sum_{j=1,2}|||u_1-u_2||u_j|||_{L_T^{\frac{4}{3}}L_x^4}+|||u_1-u_2|(e^{4\pi (1+\ep)|u_j|^2}-1)|||_{L_T^{\frac{4}{3}}L_x^4}.
\end{align*} The first term is straightforward to bound, by H\"older and Sobolev Embedding we obtain that \[|| |u_1-u_2||u_j|||_{L_T^{\frac{4}{3}}L_x^4}\le T^{\frac{3}{4}}||u_1-u_2||_{L_T^\infty L_x^8}||u_j||_{L_T^\infty L_x^8}\le T^{\frac{3}{4}}||u_1-u_2||_{X_T}||u_j||_{X_T}.\]We deal with the exponential term exactly as we did before with the same choice of $\delta=\delta(A)$ to obtain that 
\begin{align*}|||u_1-u_2|(e^{4\pi (1+\ep)|u_j|^2}-1)|||_{L_T^{\frac{4}{3}}L_x^4}&\lesssim ||u_1-u_2||_{L_T^\infty L_x^{\frac{1}{\delta}}}||e^{4\pi (1+\ep)|u_j|^2}-1||_{L_T^\frac{4}{3} L_x^\frac{4}{1-4\delta}}\\ &\lesssim_{A,M} ||u_1-u_2||_{X_T}||u_j||_{L_t^\infty L_{x}^2}^{\frac{1-4\delta}{2}}T^{\frac{3-\gamma}{4}}(T^{\frac{1}{4}}+||u_j||_{L_T^4H_{x}^{1,4}})^{\gamma}.\end{align*} Putting these together we see that \[||(II)+(IV)||_{L_T^1L_{x}^2}\lesssim_{A,M} ||u_1||_{X_T}||u_1-u_2||_{X_T}\sum_{j=1,2} T^{\frac{3}{4}}||u_j||_{X_T}+T^{\frac{3-\gamma}{4}}||u_j||_{X_T}^{\frac{1-4\delta}{2}}(T^{\frac{1}{4}}+||u_j||_{X_T})^{\gamma},\] which yields the conclusion the theorem. 
\end{proof}
\end{corollary}
\section{Well-Posedness Revisited}
In this section we revisit the well-posedness theory for $\eqref{PDE}$ and $\eqref{PDE2}$ first proved in \cite{MR2568809} using the techniques developed in the previous section.  We prove 
\begin{theorem}[\cite{MR2568809}] Suppose $u_0\in H_x^1(\R^2)$ and $||u_0||_{\dot{H}_x^1}<1$. Then there exists $T>0$ and a strong solution to $\eqref{PDE}$ and $\eqref{PDE2}$ in $C([-T,T], H_x^1(\R^2))$. 
\begin{proof}We give a proof by contraction mapping. For $T>0$ let \begin{align*}X_T&:=C([-T,T], H_x^1(\R^2))\cap L_x^{4}([-T,T], H_x^{1,4}(\R^2))\end{align*} and note that $(X_T,||\cdot ||_{X_T})$ is a Banach space. Thus, for $T,\ep>0$ \begin{align*}B(T,\ep)&=\{u\in X_T: ||u-e^{it\Delta}u_0||_{X_T}\le \ep\}\end{align*} is a complete metric space. For what follows, we assume that $0<\ep<1-||\grad u_0||_{L_x^2}$. For $u\in B(T,\ep)$ define \[\Phi_i(u)=e^{it\Delta}u_0-i\int_{0}^{t}e^{i(t-s)\Delta}F_i(u(s))\ ds.\] It suffices to prove that $\Phi_i$ is a self-map and has a fixed point. 

We first prove that $\Phi_i$ is a self-map. To this end, we first note that since the $\dot{H}_x^1$ is a conserved quantity for the free Schr\"odinger equation it follows that  \[||u||_{L_t^\infty \dot{H}_x^1}\le ||e^{it\Delta}u_0||_{L_t^\infty \dot{H}_x^1}+||u-e^{it\Delta}u_0||_{L_t^\infty \dot{H}_x^1}\le ||u_0||_{\dot{H}_x^1}+\ep<1\] for any $u\in B(T,\delta)$. Similarly, if $u\in B(T,\ep)$ then \[||u||_{X_T}\le \ep+C_{ST}||u_0||_{H_x^1},\] where $C_{ST}$ is the constant appearing in the Strichartz inequality. 
Thus, by virtue of the Strichartz Inequality and Corollaries $\ref{NLH1}$ and $\ref{NLH2}$ there exist $0<\gamma<3$ and $\delta>0$ so that \begin{align*}||\Phi_i(u)-e^{it\Delta}u_0||_{X_T}\lesssim ||F_i(u)||_{L_T^1H_x^1}&\lesssim T^{\frac{3-\gamma}{4}}(||u||_{X_T}^{\frac{3}{2}}+||u||_{X_T}^{\frac{3-4\delta}{2}}+||u||_{X_T}^{\frac{5-4\delta}{2}})(T^{\frac{1}{4}}+||u||_{X_T})^{\gamma}\\ &+T^{\frac{3}{4}}(||u||_{X_T}^2+||u||_{X_T}^3). \\ &\lesssim C(||u_0||_{H_x^1})(T^{\frac{3}{4}}+T^{\frac{3-\gamma}{4}}).\end{align*}Thus, provided $T=T(\ep)>0$ is sufficiently small we obtain that \[||\Phi_i(u)-e^{it\Delta}u_0||_{X_T}\le \ep,\] and so we obtain that $\Phi_i$ is a self-map.

We now turn to proving that $\Phi_i$ is a contraction.  Again, from Corollaries $\ref{NLH1}$ and $\ref{NLH2}$ we obtain that if $u_1,u_2\in X(T)$ then \begin{align*}||\Phi_i(u_1)-\Phi_i(u_2)||_{X_T}&\lesssim ||F_i(u_1)-F_i(u_2)||_{L_T^1 H_{x}^1}\\ &\lesssim ||u_1-u_2||_{X_T}C(||u_0||_{H_x^1})(T^{\frac{3}{4}}+T^{\frac{3-\gamma}{4}}).\end{align*} Thus, we may find $T=T(\ep)>0$ so that \[||F_i(u_1)-F_i(u_2)||_{L_T^1H_x^1}<\frac{1}{2}||u_1-u_2||_{X_T}\]
Thus, for $T=T(||u_0||_{H_x^1})>0$ sufficiently small we deduce that $\Phi_i$ is a contraction, as desired. \end{proof}
\end{theorem}
\begin{remark}The observant reader will note that the above proof proves uniqueness only in the space $X_T$, and not in the larger space $C([-T,T],H_x^1(\R^2))$. Uniqueness in $C([-T,T],H_x^1(\R^2))$ is a consequence of the techniques of the previous section, but relies heavily on the convexity of the exponential.  For details, see \cite{MR2568809}. 
\end{remark}
For completeness and the sake of exposition, we include a proof of global well-posedness. As we will see, the local theory guarantees us that we may continue a strong solution on $[0,T)$ past time $T$ provided the kinetic energy does not concentrate. Our strategy, then, is to quantify the speed at which the kinetic energy can concentrate. This is essentially a localization result due to Bourgain \cite{MR1655835} (see also Lemma 6.2 in \cite{MR1726753}). 

\begin{lemma}\label{conspeed}Let $u$ be a non-trivial solution of $\eqref{PDE}$ or $\eqref{PDE2}$ on $[0,T)$ with $0<T\le \infty$.  If $u$ solves $\eqref{PDE}$, then there exists a positive constant $C=C(M)$ so that \begin{align}
\frac{1}{\cyc{t}^2}\le C(M)(H_1(u_0)-||\grad u(t)||_{L_x^2}^2),
\end{align} for every $0<t<T$. If $u$ instead solves $\eqref{PDE2}$, then there exists a positive constant $C=C(M)$ so that \begin{align}
\frac{1}{\cyc{t}^4}\le C(M)(H_2(u_0)-||\grad u(t)||_{L_x^2}^2),
\end{align}for every $0<t<T$. 

\begin{proof}Find $R>1$ so that \begin{align}\label{Rchoice}\int_{\{|x|\le R\}}|u_0(x)|^2\ dx\ge \tfrac{1}{2}||u_0||_{L_x^2(\R^2)}^2.\end{align} 

Let $\phi:\R\to [0,1]$ be a smooth function satisfying $\phi(x)=1$ for $x\le 0$ and $\phi(x)=0$ for $x\ge 1$. Define \[\xi(x):=\phi\left(\tfrac{\text{dist}(x,B(R))}{R'}\right),\] where $\text{dist}(x,B(R))=\max\{|x|-R,0\}$ is the distance from $x$ to the ball $\{|x|\le R\}$. 

We note that $\xi(x)=1$ when $|x|\le R$, that $\xi(x)=0$ when $|x|\ge R+R'$, and that $||\grad \xi||_{L_x^\infty}\lesssim 1/R'$. This implies that \begin{align*}\int_{\{|x|\le R+R'\}}|u(t,x)|^2\ dx&\ge \int_{\R^2} |\xi(x)|^2|u(t,x)|^2\ dx,\qquad\text{and }\\\int_{\{|x|\le R\}}|u_0(x)|^2\ dx&\le \int_{\R^2} |\xi(x)|^2|u_0(x)|^2\ dx.\end{align*}By taking their difference and employing the fundamental theorem of calculus and Fubini-Tonelli, we see that
\begin{align*}\int_{\{|x|\le R+R'\}}|u(t,x)|^2\ dx-\int_{\{|x|\le R\}}|u_0(x)|^2 dx&\ge \int_{\R^2} |\xi(x)|^2(|u(t,x)|^2-|u_0(x)|^2)\ dx\\ &=\int_{\R^2}\int_{0}^{t} |\xi(x)|^2\partial_s|u(s,x)|^2\ ds\, dx\\ &=\int_{0}^{t}\int_{\R^2}|\xi(x)|^2\partial_{s}|u(s,x)|^2\ dx ds.
\end{align*}Since $u$ satisfies $\eqref{PDE}$ or $\eqref{PDE2}$, it follows that  \begin{align*}|\int_{\R^2}|\xi(x)|^2\partial_{s}|u(s,x)|^2|=|4\text{Im} \int_{\R^2}\xi(x) \grad \xi(x)\cdot \grad u(s,x) \bar{u}(s,x)|\ dx\le \tfrac{C(M)}{R'}.\end{align*} Here we used the fact that $||\grad u(t)||_{L_x^2}^2\le 1$. We deduce that  \begin{align}\label{NakanishiBound}\int_{\{|x|\le R+R'\}}|u(t,x)|^2\ dx-\int_{\{|x|\le R\}}|u_0(x)|^2 dx\ge -\tfrac{C(M)t}{R'}.\end{align} 

By expanding $e^{4\pi x^2}$ into its power series, we see that \begin{align}\label{ptmay1} 8\pi^2|u(t,x)|^4&\le e^{4\pi |u(t,x)|^2}-4\pi |u(t,x)|^2-1,\qquad \text{and }\\ \label{ptmay2} \tfrac{32\pi^3}{3}|u(t,x)|^6&\le e^{4\pi |u(t,x)|^2}-4\pi |u(t,x)|^2-8\pi^2|u(t,x)|^4-1.\end{align} This implies that \begin{align}\label{hamilbound1}
||u(t)||_{L_x^4}^{4}&\le \tfrac{1}{2\pi}(H_1(u(t))-||\grad u(t)||_{L_x^2}^2),\qquad \text{and }\\  \label{hamilbound2}||u(t)||_{L_x^6}^{6}&\le \tfrac{3}{8\pi^2}(H_2(u(t))-||\grad u(t)||_{L_x^2}^2).
\end{align}

If $u$ satisfies $\eqref{PDE}$, then by $\eqref{Rchoice}$, $\eqref{NakanishiBound}$ and $\eqref{hamilbound1}$  it follows from Cauchy-Schwarz that \begin{align*}
\tfrac{M}{2}-\tfrac{C(M)t}{R'}\le (\pi(R+R')^2)^{\frac{1}{2}}||u(t)||_{L_x^4}^{2}\lesssim (R+R')(H_1(u(t))-||\grad u(t)||_{L_x^2}^2)^\frac{1}{2}.\end{align*} Choosing $R'=\tfrac{4C(M)}{M}t$, we see from the conservation of energy that \[\tfrac{1}{\cyc{t}^2}\lesssim_{M} H_1(u_0)-||\grad u(t)||_{L_x^2}^2.\] Similarly, if $u$ satisfies $\eqref{PDE2}$, then by $\eqref{Rchoice}$, $\eqref{NakanishiBound}$ and $\eqref{hamilbound2}$ we have that
\begin{align*}
\tfrac{M}{2}-\tfrac{C(M)t}{R'}\le (\pi(R+R')^2)^{\frac{2}{3}}||u(t)||_{L_x^6}^{2}\lesssim (R+R')^\frac{4}{3}(H_2(u(t))-||\grad u(t)||_{L_x^2}^2)^\frac{1}{3}.
\end{align*} Again, choosing $R'=\tfrac{4C(M)}{M}t$ we see that \[\tfrac{1}{\cyc{t}^4}\lesssim_{M}H_2(u_0)-||\grad u(t)||_{L_x^2}^2.\]

\end{proof}
\end{lemma}

\begin{theorem}[\cite{MR2568809}] \label{GWP}If $u_0\in H_x^1$ then there exists a unique $C(\R,H_x^1(\R^2))$ solution to $\eqref{PDE}$ and $\eqref{PDE2}$ provided $H_1(u_0)\le 1$ and $H_2(u_0)\le 1$, respectively. 
\begin{proof} We only consider positive times, as the following argument can be repeated for negative times with no change. Suppose that $u$ is a strong solution to $\eqref{PDE}$ on some maximal interval $[0,T^*)$ with $T^*<\infty$. Suppose $t_n$ is a sequence of times with $t_n\nearrow T^*$. Passing to a subsequence, we may assume that $||\grad u(t_n)||_{L_x^2(\R^2)}\to L\le H_i(u_0)$. If $u$ solves $\eqref{PDE}$, then by Lemma \ref{conspeed} we know that \[\tfrac{1}{\cyc{t_n}^2}\le C(M)(H_1(u_0)-||\grad u(t_n)||_{L_x^2}^2).\] We see then that\[\lim_{n\to \infty}||\grad u(t_n)||_{L_x^2}^2=L<H_1(u_0)\le 1.\] The local theory guarantees the existence of a local solution with initial data $u(t_n)$ with a lifetime at least $\tau=\tau(L)>0$. But then for $n$ satisfying $T^*-t_n<\tau$ we may continue $u$ past $T^*$, producing a contradiction. A similar argument reaches the same conclusion if $u$ instead solves $\eqref{PDE2}$. 
\end{proof}
\end{theorem}
\section{Perturbation Theory}
\subsection{Stability Theory}
In this section we derive the perturbation theory essential to the proof of Theorem \ref{perturbtheorem}. In the energy-subcritical setting, we will see that we may treat $\eqref{PDE2}$ as a perturbation of the mass-critical NLS. To do so, we must first understand how approximate solutions to the mass-critical NLS behave. 

\begin{lemma}[Short-Time Mass-Critical Perturbations] \label{ST}Let $I$ be a compact interval and let $\tilde{v}$ be an approximate solution to $\eqref{mcnls}$ in the sense that \[(i\partial_t+\Delta)\tilde{v}=|\tilde v|^2v+e\]for some spacetime function $e$. Assume that
    \begin{align}||\tilde{v}||_{L_t^\infty H_x^1(I\times \R^2)}\le M \label{ST1}\end{align}
\noindent for some positive constant $M$. Let $t_0\in I$ and let $v(t_0)$ be close to $\tilde{v}(t_0)$ in the sense that
    \begin{align}
        ||v(t_0)-\tilde{v}(t_0)||_{H_x^1}\le M' \label{ST2},\end{align}
    for some $M'>0$. Assume also the smallness conditions

    \begin{align}
        ||\tilde{v}||_{L_t^4H_x^{1,4}}&\le \epsilon_0,\label{ST3}\\
        ||e^{i(t-t_0)\Delta}(v(t_0)-\tilde{v}(t_0))||_{L_t^4H_x^{1,4}}&\le \epsilon,\qquad \text{and}\label{ST4}\\
        ||e||_{N(I\times \R^2)}&\le \epsilon \label{ST5} \end{align}
\noindent for some $0<\epsilon\le \epsilon_0$, where $\epsilon_0=\epsilon_0(M,M')>0$ is a small constant. Then, there exists a solution $v$ to \eqref{mcnls} on $I\times \R^2$ with initial data $v(t_0)$ at time $t=t_0$ satisfying
    \begin{align}
        ||v-\tilde{v}||_{L_t^4H_x^{1,4}}&\lesssim \epsilon, \label{ST6}\\
        ||v-\tilde{v}||_{S(I\times \R^2)}&\lesssim M', \label{ST7}\\
        ||v||_{S(I\times \R^2)}&\lesssim M+M', \qquad \text{and}\label{ST8}\\
        ||(i\partial_t+\Delta)(v-\tilde{v})+e||_{N(I\times \R^2)}&\lesssim \epsilon \label{ST9}.\end{align}\end{lemma}
\begin{remark}
Note that by Strichartz, hypothesis $\eqref{ST4}$ is redundant if $M'=O(\epsilon)$.
\end{remark}
\begin{proof}
    By time symmetry we may, and do, assume that $t_0=\inf I$. Let $z:=v-\tilde{v}$. Then $z$ satisfies \[\left\{\begin{array}{ll}(i\partial_t+\Delta)z=|\tilde{v}+z|^2(\tilde{v}+z)-|\tilde{v}|^2\tilde{v}-e\\ z(t_0)=v(t_0)-\tilde{v}(t_0)\end{array}  \right.\] For $t\in I$ define \[S(t):=||(i\partial_t +\Delta)z+e||_{{N}([t_0,t]\times \R^2)}=|||\tilde{v}+z|^2(\tilde{v}+z)-|\tilde{v}|^2\tilde{v}||_{{N}([t_0,t]\times \R^2)}.\] We will use the pointwise estimates $\eqref{cubicdiff}$ and $\eqref{cubicderivdiff}$ to estimate the $N^0$ and $N^1$ norms respectively.  Indeed, with $g(z)=z|z|^2$, we have
    \begin{align*}
    |g(\tilde{v}+z)-g(\tilde{v})|&\lesssim |z|(|z|^2+|\tilde{v}|^2)\\ 
    |\grad g(\tilde{v}+z)-\grad g(\tilde{v})|&\lesssim |\tilde{v}|^2|\grad z|+|z|^2|\grad \tilde{v}|+|z|^2|\grad z|+|\grad \tilde{v}||z||\tilde v|.
    \end{align*}
    Since $(\tfrac{4}{3},\tfrac{4}{3})$ is a dual Strichartz pair, we know for any three spacetime functions $f_1(t,x)$, $f_2(t,x)$, and $f_3(t,x)$ that \begin{align}\label{mastic}||f_1f_2f_3||_{N^0}\lesssim ||f_1f_2f_3||_{L_t^{\frac{4}{3}}L_x^{\frac{4}{3}}}\lesssim ||f_1||_{L_t^4L_x^4}||f_2||_{L_t^4L_x^4}||f_3||_{L_t^4L_x^4}.\end{align}Repeated applications of \eqref{mastic} implies, by ($\ref{ST3}$), that \begin{align}
        S(t)& \lesssim ||z||_{L_t^4H_{x}^{1,4}}^3+||z||_{L_t^4H_{x}^{1,4}}^2||v||_{L_t^4H_{x}^{1,4}}+||\tilde v||_{L_t^4L_t^4}^2||z||_{L_t^4H_x^{1,4}}\\ &\lesssim  ||z||_{L_t^4H_{x}^{1,4}}^3+\ep_0^2||z||_{L_t^4H_{x}^{1,4}}^2+\ep_0||z||_{L_t^4H_x^{1,4}}\end{align}
    On the other hand by Strichartz, ($\ref{ST4}$), and ($\ref{ST5}$), we have \begin{align}||z||_{L_t^4H_x^{1,4}}\lesssim ||e^{i(t-t_0)\Delta}z(t_0)||_{L_t^4 H_x^{1,4}}+S(t)+||e||_{N([t_0,t]\times \R^2)}\lesssim S(t)+\epsilon\label{ST11}.\end{align} So we have that \[S(t)\lesssim (S(t)+\ep)^3+\ep_0^2(S(t)+\ep)^2+\ep_0(S(t)+\ep).\]A continuity argument shows then that if $\epsilon_0$ is taken suffiiently small, then \[S(t)\le \epsilon\text{ for any }t\in I.\] This implies ($\ref{ST9}$). Using ($\ref{ST9}$) and
    ($\ref{ST11}$), one easily derives ($\ref{ST6}$). Moreover, by Strichartz, ($\ref{ST2}$), ($\ref{ST5}$), and ($\ref{ST9}$),
    \[||z||_{S(I\times \R^2)}\lesssim ||z(t_0)||_{H_x^1}+||(i\partial_t+\Delta)z+e||_{N(I\times \R^2)}+||e||_{N([t_0,t]\times\R^2)}\lesssim M'+\epsilon\] which proves ($\ref{ST7}$). To prove ($\ref{ST8}$), we use Strichartz, ($\ref{ST1}$), ($\ref{ST2}$), ($\ref{ST3}$), ($\ref{ST5}$), and ($\ref{ST9}$) we have\begin{align*}
        ||v||_{S}&\lesssim ||v(t_0)||_{H_x^1}+||(i\partial_t+\Delta)v||_{N}\\ &\lesssim ||\tilde{v}(t_0)||_{H_x^1}+||v(t_0)-\tilde{v}(t_0)||_{H_x^1}+||(i\partial_t+\Delta)(v-\tilde{v})+e||_{N}+||(i\partial_t+\Delta)\tilde{v}||_{N}+||e||_{N}\\ &\lesssim M+M'+\epsilon+||(i\partial_t+\Delta)\tilde{v}||_{ L_t^\frac{4}{3}L_x^\frac{4}{3}}\\ &\lesssim M+M'+||\tilde{v}||_{L_t^4L_x^4}^3\\ &\lesssim M+M'+\epsilon_0^3\\ &\lesssim M+M'.
    \end{align*}\end{proof}
    The following proposition is an expected consequence of the previous lemma. By the usual combinatorial argument, it suffices to prove the following proposition with $\eqref{SL2}$ replaced by $||\tilde{v}||_{L_t^4H_{x}^{1,4}}<\ep\le \ep_0$, with $\ep_0$ as in the previous lemma. 
\begin{proposition}\label{SL}
    Let $I$ be a compact interval and let $\tilde{v}$ be an approximate solution in the sense that \[(i\partial_t+\Delta)\tilde{v}=|\tilde v|^2v+e\]for some function $e$. Let $\epsilon_0$ be as in the previous lemma. Assume that \begin{align}||\tilde{v}||_{L_t^\infty H_x^1(I\times \R^2)}&\le M \label{SL1}\\ ||\tilde{v}||_{L_t^4H_x^{1,4}}\le L,\label{SL2}\end{align} for some positive constant $M$ and $L$. Let $t_0\in I$ and let $v(t_0)$ be close to $\tilde{v}(t_0)$ in the sense that
    \begin{align}||v(t_0)-\tilde{v}(t_0)||_{H_x^1}\le \ep,\label{SL3}\end{align}
    for some $\epsilon_0>\epsilon>0$. Assume also the smallness condition \begin{align}||e||_{N(I\times \R^2)}\le \epsilon \label{SL5},
    \end{align}for some $0<\epsilon\le \epsilon_1$ where $\epsilon_1=\epsilon_0(M)>0$ is a small constant. Then, there exists a solution $v$ to $\eqref{mcnls}$ on $I\times \R^2$ with initial data $v(t_0)$ at time $t=t_0$ satisfiying \begin{align}
        ||v-\tilde{v}||_{S(I\times \R^2)}&\le C(M,L)\epsilon. \label{SL7}\end{align}

\end{proposition}

\subsection{Proof of Theorem $\ref{perturbtheorem}$}
In this section, we upgrade the {\em a priori} $L_t^4L_x^8$ bound granted by Morawetz to derive global spacetime bounds that imply scattering for $\eqref{PDE}$ in the energy-subcritical case. As in \cite{MR2354495}, our approach is perturbative. We split the nonlinear term in $\eqref{PDE}$   \begin{align*}&\left\{\begin{array}{ll}iu_t+\Delta u=4\pi|u|^2u+F_2(u)\\ u(0,x)=u_0\in
        H_x^1(\R^2)\end{array}\right.\end{align*}and view $\eqref{PDE}$ as a perturbation to the mass-critical NLS with $F_2$ as an error term. The dual Strichartz estimates on $F_2$ derived from Corollary \ref{subcritmora} will grant us sufficiently good bounds to use the stability theory.
\begin{proof}[Proof of Theorem \ref{perturbtheorem}]Let $u$ be the global solution to $\eqref{PDE}$ given by Theorem $\ref{GWP}$.  By Lemma $\ref{morawetz}$ and the conservation of energy and mass, we see that \[||u||_{L_t^4L_x^8}\lesssim ||u||_{L_t^\infty H_x^1(\R\times \R^2)}\le C(H,M).\] Let $\epsilon$ be a small constant to be chosen later. Split $\R$ into $J=J(H,M,\epsilon)$ subintervals $I_j$, $0\le j\le J-1$ so that \[||u||_{L_t^4L_x^8(I_j)}\sim \epsilon.\] We'll show that $u$ obeys good Strichartz bounds on each slab $I_j\times \R^2$. It follows from Corollary \ref{subcritmora} that \begin{align}\label{PF1}||e||_{{N}(I_j\times \R^2)}\le C(H,M)\epsilon^{4\delta}||u||_{L_t^4H_x^{1,4}(I_j)}^{2},\end{align}for all $0\le j\le J-1$. In what follows, we fix an interval $I_{j_0}=[a,b]$ and prove that $u$ obeys good Strichartz estimates on the slab $I_{j_0}\times \R^2$. In order to do so, we view the solution $u$ as a perturbation to a solution to $\eqref{mcnls}$ \[\left\{\begin{array}{ll}(i\partial_t+\Delta)v=4\pi|v|^2v\\ v(a)=u(a).\end{array}\right..\] This initial value problem is globally well-posed in $H_x^1$ by Theorem \ref{dodson}, and the unique solution enjoys the global spacetime bound \[||v||_{S}\le C(H,M).\] Thus, we can subdivide $\R$ into $K=K(M,H,\eta)$ subintervals $J_k$ so that \begin{align}\label{PF2}||v||_{L_t^4H_x^{1,4}(J_k)}\sim \eta,\end{align} on each $J_k$, where $\eta>0$ is a small constant to be chosen later. Without loss of generality, we assume that \[ [a,b]=\cup_{k=0}^{K'-1}J_{k}, t_0=a, t_{K'}=b\] The nonlinear evolution $v$ being small on $J_k\times \R^2$ implies that the linear evolution $e^{i(t-t_k)\Delta}v(t_k)$ is small as well. By Strichartz, we have that \begin{align*}||e^{i(t-t_k)\Delta}v(t_k)||_{L_t^4H_x^{1,4}(J_k)}&\le
        ||v||_{L_t^4H_x^{1,4}}+C_{ST}|||v|^{2}v||_{ L_t^\frac{4}{3}L_x^\frac{4}{3}}+C_{ST}||\nabla|v|^{2}v||_{ L_t^\frac{4}{3}L_x^\frac{4}{3}}\\ &\le   ||v||_{L_t^4H_x^{1,4}}+C_{ST}||v||_{L_t^4L_x^4}^2||v||_{L_t^4L_x^4}+C_{ST}||v||_{L_t^4L_x^4}^3\\ &\le \eta+2C_{ST}\eta^{3},\end{align*} where $C_{ST}$ is the constant appearing in the Strichartz inequality. If $\eta<(4C_{ST})^{-1/2},$ then \begin{align}\label{PF3}||e^{i(t-t_k)\Delta}v(t_k)||_{L_t^4 H_x^{1,4}(J_k)}\le 2\eta.\end{align} 
        
        Our goal is to apply the our mass-critical stability theory to compare $u$ and $v$ on the slab $[t_0,t_1]\times \R^2$. To do so, we must verify that $u$ and $v$ satisfy the hypotheses of Proposition \ref{ST}. Once we do, this will guarantee that $u$ and $v$ remain close at time $t_1$. We will use this closeness to verify the hypotheses of Proposition \ref{ST} for $u$ and $v$ on the slab $[t_1,t_2]\times \R^2$, to show that they continue to remain close on the entirety of the slab. We continue in this fashion iteratively, recycling the closeness on one slab to verify the hypotheses of Proposition \ref{ST} for $u$ and $v$ on the next, to show that $u$ and $v$ remain close on all of $I_{j_0}\times \R^2$. 
        
To establish the base case, recall that $u(t_0)=v(t_0)$. So, by Strichartz, $\eqref{PF3}$, and $\eqref{PF1}$ we know that \begin{align*}
        ||u||_{L_t^4H_x^{1,4}([t_0,t_1])}&\le ||e^{i(t-t_0)\Delta}u(t_0)||_{L_t^4H_x^{1,4}([t_0,t_1])}+C_{ST}|||u|^2u||_{N([t_0,t_1])}+C_{ST}||F_2(u)||_{N([t_0,t_1])}\\&\le 2\eta+C_{ST}||u||_{L_t^4H_x^{1,4}([t_0,t_1])}^3+C(H,M)\epsilon^{4\delta}||u||_{L_t^4H_x^{1,4}([t_0,t_1])}^{2}\end{align*} which, by a continuity argument, yields that \begin{align}\label{PF4}||u||_{L_t^4 H_x^{1,4}}\le 4\eta\end{align}provided that $\eta$ and $\ep=\ep(H,M)$ are sufficiently small. To apply our stability lemma, we need to check that the error term $F_2(u)$ is small
    in $N(J_0\times \R^2)$. But this is straightforward since, by Corollary \ref{subcritmora},  \begin{align}\label{PF5}||F_2(u)||_{{N([t_0,t_1])}}\le C(H,M)\epsilon^{4\delta}||u||_{L_t^4H_x^{1,4}(J_0\times \R^2)}^2\le  C(H,M)\eta^{2} \epsilon^{4\delta}.\end{align} Choosing $\ep$ sufficiently small depending only on $H$ and $M$, we have that \[||u-v||_{{S}(J_0\times \R^2)}\le \epsilon^{2\delta}.\] By Strichartz, this implies that 
    \begin{align}            \label{PF6}||u(t_1)-v(t_1)||_{H_x^1}\le \ep^{2\delta},\\
   \label{PF7} ||e^{i(t-t_1)\Delta}(u(t_1)-v(t_1))||_{L_t^4H_{x}^{1,4}}\lesssim \ep^{2\delta}
    \end{align}
    Now we'll use ($\ref{PF6}$) and ($\ref{PF7}$) to estimate $u$ on the slab
    $J_1\times \R^2$. Splitting the linear evolution, we see by Strichartz, ($\ref{PF1}$),
    ($\ref{PF3}$), and ($\ref{PF7}$) that
        \begin{align*}
            ||u||_{L_t^4H_x^{1,4}(J_1)}&\le ||e^{i(t-t_k)\Delta}v(t_k)||_{L_t^4H_x^{1,4}(J_1)}+
            ||e^{i(t-t_k)\Delta}(u(t_k)-v(t_k))||_{L_t^4H_x^{1,4}(J_1)}\\
            &\qquad +C_{ST}||u||_{L_t^4H_x^{1,4}(J_1)}^3+C(H,M)\epsilon^{4\delta}||u||_{L_t^4H_x^{1,4}(J_1)}^2\\
            &\le 2\eta+C_{ST}\epsilon^{2\delta}+C_{ST}||u||_{L_t^4H_x^{1,4}(J_1)}^{3}
            +C(H,M)\epsilon^{4\delta}||u||_{L_t^4H_x^{1,4}(J_1)}^2.
        \end{align*}

 A standard continuity argument then yields that $||u||_{L_t^4H_x^{1,4}(J_1)}\le 4\eta$ provided $\eta$ and $\epsilon=\epsilon(H,M)$ are chosen sufficiently small.
 
 This implies that \[||F_2(u)||_{N(J_1\times \R^2)}\le C(H,M)\eta^2 \ep^{4\delta}.\]Choosing $\epsilon$ sufficiently small depending only on $H$ and $M$, we can apply the stability lemma to derive that \[||u-v||_{S(J_1\times \R^2)}\le \ep^{\delta}.\]  By induction, if we take $\ep=\ep(H,M)$ smaller at each step, we obtain that \begin{align}\label{induct}
        ||u-v||_{S(J_k\times \R^2)}\le \ep^{\delta/2^{k-1}}.
        \end{align} Adding these bounds over all intervals $J_k$ which meet $I_{j_0}$ we deduce that \begin{align}\label{derp}\sum_{j=0}^{K'-1}||u||_{S([t_j,t_{j+1}]\times \R^n)}&\lesssim ||v||_{S(I_{j_0}\times \R^2)}+\sum_{k=0}^{K'-1}||u-v||_{S(J_k\times \R^n)}\le C(H,M).\end{align}Thus, we find that \eqref{derp} holds with $K'$ in place of $k$. Since $I_{j_0}$ was arbitrarily chosen, we may some over the finitely many $J$ intervals to see that $||u||_{S(\R\times \R^2)}\le JC(H,M)<\infty$, which completes the proof. \end{proof}
 \subsection{Finite global Strichartz norms imply scattering}
 We show that finite global Strichartz norms imply scattering. We will only present the construction of the scattering states, and demonstrate that their linear flow asymptotically approximates our solution. In fact, we will only construct scattering states in the positive time direction, since identical arguments can be used in the negative time direction. Standard techniques can be used to construct wave operators, see \cite{MR2002047}. 
 
Suppose $u_0\in H_x^1$ with $H(u_0)=H<1$ and that $u(t)$ is a global solution to \eqref{PDE}. Then, by Theorem \ref{perturbtheorem}, we know that \[||u||_{L_t^4H_{x}^{1,4}}\le C(H,M).\] For $0<t<\infty$ define \[u_{+}(t)=u_0-i\int_{0}^{t}e^{-is\Delta}F_1(u(s))\ ds.\] By the Strichartz inequality, we see that \[||u_{+}||_{L_t^\infty H_x^1}\lesssim ||u_0||_{H_x^1}+||F_1(u)||_{N(\R)}\lesssim ||u_0||_{H_x^1}+||4\pi |u|^2u||_{L_t^\frac{4}{3}H_{x}^{1,\frac{4}{3}}}+||F_2(u)||_{L_t^\frac{4}{3}H_{x}^{1,\frac{4}{3}}}.\] By Corollary \ref{subcritmora}, there exists $0<\delta=\delta(H)<\frac{1}{3}$ so that\begin{align*}
||4\pi |u|^2u||_{L_t^\frac{4}{3}H_{x}^{1,\frac{4}{3}}}\lesssim ||u||_{L_t^4H_{x}^{1,4}}^3<\infty\qtq{and}||F_2(u)||_{N}\lesssim ||u||_{L_t^4L_x^8}^{4\delta}||u||_{L_t^4H_{x}^{1,4}}^{2}<\infty.
\end{align*} Thus $u_+(t)\in H_x^1$ for all $t\in \R$. A similar argument shows that $u_{+}(t)$ converges in $H_x^1$ as $t\to \infty$. Define \[u_{+}=u_0-i\int_{0}^{\infty}e^{-is\Delta}F_1(u(s))\ ds.\]
With $u_{+}$ so defined, we see that \begin{align*}
||e^{it\Delta}u_{+}-u(t)||_{H_x^1}&\lesssim ||\int_{t}^{\infty}e^{-is\Delta}F_1(u(s))\ ds||_{H_x^1}\\&\lesssim ||u||_{L_t^4H_{x}^{1,4}([t, \infty]\times \R^2)}^2+||u||_{L_t^4L_x^8([t, \infty]\times \R^2)}^{4\delta}||u||_{L_t^4H_{x}^{1,4}([t, \infty]\times \R^2)}^{2}
\end{align*}which tends to $0$ as $t\to \infty$.

\bibliography{NLScattering}
\end{document}